\documentclass[11pt]{amsart}
\usepackage{graphicx}
\usepackage{euscript}
\usepackage{amsmath}
\usepackage{amsfonts}
\usepackage{amssymb}
\usepackage{amsthm}
\usepackage{epsfig}
\usepackage{epstopdf}
\usepackage{url}
\usepackage{array}
\usepackage{amssymb}\usepackage{amscd}
\usepackage{epic}
\usepackage{eufrak}
\usepackage{tikz,underscore}
\usepackage{tikz-cd}
\usepackage{float}

\numberwithin{equation}{section}

\theoremstyle{plain}
\newtheorem{theorem}{Theorem}[section]
\newtheorem{lemma}[theorem]{Lemma}
\newtheorem{proposition}[theorem]{Proposition}

\newtheorem{corollary}[theorem]{Corollary}

\newtheorem{ques}[theorem]{Question}

\theoremstyle{definition}
\newtheorem{example}[theorem]{Example}
\newtheorem{remark}[theorem]{Remark}
\newtheorem{definition}[theorem]{Definition}

\newcommand{\bH}{{\mathbb H}}

\def\Q{\mathbb Q}
\def\R{\mathbb R}
\def\Z{\mathbb Z}

\def\F{\mathbb F}

\title{Heegaard Floer Homology Of $L$-space Links With Two Components}

\author{Beibei Liu}
\address{Department of Mathematics, UC Davis, One shields Avenue, Davis CA 95616, USA}
\email{bxliu@math.ucdavis.edu}

\begin{document}

\begin{abstract}
We compute different versions of link Floer homology $HFL^{-}$ and $\widehat{HFL}$ for any $L$-space link with two components. The main approach is to compute the $h$-function of the filtered chain complex which is determined by the Alexander polynomials of every sublink of the $L$-space link. As an application, Thurston polytope and Thurston norm of any 2-component $L$-space link are explicitly determined by Alexander polynomials of the link and the link components. 
\end{abstract}

\maketitle

\section{Introduction}
In this article, we present a general method to compute the Heegaard Floer link homology for any $L$-space link with two components. Usually, it is very hard to compute the Heegaard Floer link homology $HFL^{-}$ and $\widehat{HFL}$. For $L$-space links, the computation is easier. Yajing Liu computed the link Floer homology $HFL^{-}$ for any $L$-space link with two components \cite{Liu}. We revisit his computation and compute link Floer homology $\widehat{HFL}$ of any $L$-space link with two components. As a consequence, we  compute the Thurston polytope and Thurston norm of the link complement.  Recall that for an $L$-space link with $r$ components with a given generic admissible multipointed Heegaard diagram, we can associate the \emph{generalized Floer complexes} $A^{-}(\mathbf{s})$  filtered by Alexander gradings. Here we work over $\F=\F_{2}$  and $\mathbf{s}\in \bH$ where $\bH$ is some $r$-dimensional lattice (see Definition 2.3 and \cite{MO}). If the link $L$ is an $L$-space link, there is an important result for the generalized Floer complexes:

\begin{proposition}\cite[Proposition 1.11]{Liu}
For any $L$-space link, $H_{\ast}(A^{-}(\mathbf{s}))=\F[[U]]$ with $\mathbf{s}\in \bH$.
\end{proposition}

Here $U$ has homological grading $-2$. Define $-2h(\mathbf{s})$ as the homological grading of the unique generator in $H_{\ast}(A^{-}(\mathbf{s}))$. By the work of E. Gorsky, A. N\'emethi and Yajing Liu,  $h(\mathbf{s})$ is determined by  the Alexander polynomials $\Delta_{L}(t_{1}, t_{2}), \Delta_{L_{1}}(t)$ and $\Delta_{L_{2}}(t)$ for any $L$-space link with two components $L=L_{1}\cup L_{2}$ and $\mathbf{s}\in \bH$. For any 2-component $L$-space link,   there is  a spectral sequence which converges to $HFL^{-}(L, \mathbf{s})$ \cite{GN}. In the spectral sequence, the $E^{1}$-page is combinatorially determined by $h(\mathbf{s})$ and the spectral sequence  collapses at $E^{2}$-page \cite[Theorem 2.2.10]{GN} \cite{Liu}. 

The computation of $\widehat{HFL}(L, \mathbf{s})$ is more complicated. We introduce a bigraded ``iterated cone" complex  $(\mathfrak{C}(s_{1}, s_{2}), d+d_{1})$ in Section 3. There exists a spectral sequence associated to this bigraded complex where $E^{1}$-page is defined by $HFL^{-}$ and $E^{3}=\widehat{HFL}(L, s_{1}, s_{2})$.  Theorem \ref{thm3} shows that the $E^{1}$-page of this spectral sequence is $HFL^{-}(s_{1}+1, s_{2}+1)\oplus HFL^{-}(s_{1}, s_{2}+1)\oplus HFL^{-}(s_{1}+1, s_{2})\oplus HFL^{-}(s_{1}, s_{2})$ and the differential $d_{1}$ is induced by the actions of $U_{1}$ and $U_{2}$. Lemma \ref{U-action} in Section 3 indicates how  $U_{i}$ act on  the Heegaard Floer link homology $HFL^{-}(L, \mathbf{s})$ for any $\mathbf{s}\in \bH$. So we can compute $E^{2}$-page of the spectral sequence. If $d_{2}=0$, the spectral sequence collapses at $E^{2}$-page. If $d_{2}$ is nonzero, we need to use another strategy to compute $\widehat{HFL}(L, \mathbf{s})$. We first find all possible cases where $d_{2}$ may be nontrivial. In order to compute $\widehat{HFL}(L, \mathbf{s})$, we  use the symmetry of Heegaard Floer link homology:  $\widehat{HFL}(L, \mathbf{s})\cong \widehat{HFL}(L, \mathbf{-s})$ up to some grading shift  \cite[Equation 5]{os3}. In  Section 3, we find that in all cases where $d_{2}$ may be nontrivial, the  differential $d_{2}$ must vanish in the spectral sequence corresponding to $\widehat{HFL}(L, \mathbf{-s})$. Then we can compute $\widehat{HFL}(L, \mathbf{-s})$ and therefore $\widehat{HFL}(L, \mathbf{s})$. Thus  we  compute $\widehat{HFL}$ for all $L$-space links with two components and obtain the main theorem of this paper. 

\begin{theorem}
\label{thm 2}
For any $L$-space link $L=L_{1}\cup L_{2}$ with two components, $\widehat{HFL}(s_{1}, s_{2})$ is determined by the $h$-function and hence determined by the symmetrized Alexander polynomials $\Delta_{L}(t_{1}, t_{2}), \Delta_{L_{1}}(t)$ and $\Delta_{L_{2}}(t)$ and the linking number $lk$ of components $L_{1}$ and $L_{2}$. 
\end{theorem}

\begin{remark}
Heegaard Floer link homology depends on the orientation of the link. For any $L$-space link $L=L_{1}\cup L_{2}$, we need to give this link an orientation which will determine the linking number of the two link components $L_{1}$ and $L_{2}$. 
\end{remark}

Yajing Liu \cite{Liu} showed that rank$_{\F}(HFL^{-}(L, \mathbf{s}))\leq 2$. We show that $4$ is a bound for the rank of link Floer homology $\widehat{HFL}$ for any $L$-space link with two components and give examples for all possible ranks from $0$ to $4$ in Section 3. 

\begin{corollary}
\label{rank}
For any $L$-space link $L=L_{1}\cup L_{2}$ with two components and any  $\mathbf{s}\in \bH$, $\textup{rank}_{\F}(\widehat{HFL}(L, \mathbf{s}))\leq 4$. In particular, $\mid \chi(\widehat{HFL}(L, \mathbf{s}))\mid \leq 4$. 
\end{corollary}

In Section 4, we present an application of Theorem \ref{thm 2}. It is known from the work of  P.Ozsv\'ath and Z.Szab\'o \cite{oss} that $\widehat{HFL}(L)$ detects the Thurston norm of the link complement. Recall that for any compact, oriented surface with boundary (maybe disconnected)  $F=\bigcup _{i=1}^{n} F_{i}$, define its \emph{complexity} as $$\chi_{-}(F)=\sum\limits_{\lbrace F_{i}\mid \chi(F_{i})\leq 0\rbrace} -\chi(F_{i}).$$ 
For any link $L\subseteq S^{3}$, and any homology class $h\in H_{2}(S^{3}, L)$, there exists a compact oriented surface $F$ with boundary embedded in $S^{3}-\textup{nd}(L)$ which represents this homology class (i.e $[F]=h$). So for any homology class $h\in H_{2}(S^{3}, L; \Z)$, we can assign a function:
$$x(h)=\min \limits_{F\hookrightarrow S^{3}-\textup{nb}(L), [F]=h} \chi_{-}(F).$$
This function can be naturally extended to  a semi-norm, the \emph{Thurston semi-norm}, denoted by $x: H_{2}(S^{3}, L; \R)\rightarrow \R$ \cite{oss}. The unit ball for the norm $x$ is called \emph{Thurston polytope}. By computing the convex hull of $\mathbf{s}\in \bH$ where $\widehat{HFL}(L, \mathbf{s})\neq 0$, which is also called \emph{link Floer homology polytope}, we can compute the dual Thurston polytope and thus Thurston norm by the work of  P.Ozsv\'ath and Z.Szab\'o \cite{oss}.  So the Thurston polytope and Thurston norm for any $2$-component $L$-space link $L=L_{1}\cup L_{2}$ are determined by Alexander polynomials of every sublink, but in a very nontrivial way.

\begin{theorem}
If $L=L_{1}\cup L_{2}$ is an $L$-space link with two components in $S^{3}$, then the Thurston norm of the link complement is determined by the Alexander polynomials $\Delta_{L}(t_{1}, t_{2}), \Delta_{L_{1}}(t) , \Delta_{L_{2}}(t)$ and the linking number of the two components $L_{1}$ and $L_{2}$. 
\end{theorem}

P.Ozsv\'ath and Z.Szab\'o  point out that for any alternating link, its  Thurston polytope is dual to the Newton polytope of the multi-variable Alexander polynomial \cite{oss} and McMullen showed that the Newton polytope of the multi-variable Alexander polynomial of any link is contained in its dual Thurston polytope \cite{Mu} .  We compute the dual Thurston polytopes of two non-alternating $L$-space links with two components in Examples 4.3 and 4.4. They  both agree with the Newton polytopes of Alexander polynomials. A natural question arises: 

\begin{ques}
\label{ques}
For any $L$-space link with two components which is not a split union of two $L$-space knots, is the  Thurston polytope dual to the Newton polytope of its multi-variable Alexander polynomial?
\end{ques}

\begin{remark}
In Example 4.3, we present a $2$-component $L$-space link where supp$(\widehat{HFL})=\lbrace(s_{1}, s_{2})\in \bH | \widehat{HFL}(s_{1}, s_{2})\neq 0   \rbrace$ is larger than supp$(\chi(\widehat{HFL}))=\lbrace(s_{1}, s_{2})\in \bH \mid \chi(\widehat{HFL}(s_{1}, s_{2})) \neq 0    \rbrace$. But the convex hull of supp$(\widehat{HFL})$ is the same with the convex hull of supp$(\chi(\widehat{HFL}))$ since the lattice points $(s_{1}, s_{2})\in \bH$ for which $\chi(\widehat{HFL}(s_{1}, s_{2}))=0$ and $\widehat{HFL}(s_{1}, s_{2})\neq 0$ are inside of the convex hull of supp$(\chi(\widehat{HFL}))$. 
\end{remark}

For any split $L$-space link, the answer to Question \ref{ques} is negative since its Alexander polynomial vanishes, but its dual Thurston polytope is nonempty. Example 5.5 gives the link Floer homology polytope of the split union of two right-handed trefoils. Recall that the split union of two $L$-space knots is an $L$-space link \cite{Liu}, and the $h$-function of the link satisfies that  $h(s_{1}, s_{2})=h_{1}(s_{1})+h_{2}(s_{2})$ where $(s_{1}, s_{2})\in \bH$ and $h_{1}, h_{2}$ are $h$-functions of the link components $L_{1}$ and $L_{2}$ respectively. We compute $\widehat{HFL}$ for any split union of two $L$-space knots. In general, we compute $\widehat{HFL}$ for all $2$-component $L$-space links with Alexander polynomial $\Delta(t_{1}, t_{2})=0$. 

\begin{theorem}
\label{split}
For any $L$-space link $L=L_{1}\cup L_{2}$ with two components, if $\Delta_{L}(t_{1}, t_{2})=0$, $\widehat{HFL}(L, s_{1}, s_{2})\cong \widehat{HFL}(L_{1}\sqcup L_{2}, s_{1}, s_{2})\cong \widehat{HFL}(L_{1}, s_{1})\otimes \widehat{HFL}(L_{2}, s_{2})\otimes (\F\oplus \F_{-1})$ where $L_{1}\sqcup L_{2}$ denotes the split union of $L_{1}$ and $L_{2}$ and $(s_{1}, s_{2})\in \bH$. 

\end{theorem}

\section*{ACKNOWLEDGEMENTS}
I  deeply appreciate Eugene Gorsky for introducing this interesting topic to me and his patiently teaching on Heegaard Floer homology, and also for his constant guidance and discussions during the project. I am also grateful to Allison Moore,  Yi Ni and Jacob Rasmussen for the useful discussions on $L$-space links. The paper is inspired by the work of Yajing Liu, and the project is partially supported by NSF-1559338.

\section{Heegaard Floer Link Homology}
\subsection{$L$-space Links}
In \cite{OS}, P.Ozsv\'ath and Z.Szab\'o introduced the concept of $L$-space.

\begin{definition}
A 3-manifold $Y$ is an $L$-space if it is a rational homology sphere and its Heegaard Floer homology has minimal possible rank: For any spin$^{c}$-structure $s$, $\widehat{HF}(Y, s)=\F$ has rank 1 and $HF^{-}(Y, s)$ is a free $\F[U]$-module of rank 1. 
\end{definition}

In \cite{GN2}, E. Gorsky and A. N\'emethi define $L$-space links in terms of large surgeries.

\begin{definition}
An $l$-component link $L\subseteq S^{3}$ is an $L$-space link if there exist integers $p_{1}, p_{2}, \cdots, p_{l}$ such that for all integers $n_{i}\geq p_{i}, 1\leq i\leq l$, the $(n_{1}, n_{2}, \cdots, n_{l})$-surgery $S^{3}_{n_{1}, n_{2}, \cdots, n_{l}}$ is an $L$-space. 
\end{definition}

The computation of Heegaard Floer link homology is not easy. However, $L$-space links have some nice properties and these make the computation of Heegaard Floer link homology easier. In particular, we only consider $L$-space links $L=L_{1}\cup L_{2}$ with two components in this article. 

For a $2$-component $L$-space link $L=L_{1}\cup L_{2}$ on  $S^{3}$, consider a generic admissible multi-pointed Heegaard diagram with each component $L_{i}$ having only two basepoints $w_{i}, z_{i}$. Recall that in \cite[Section 4]{MO},  we can associate a generalized Floer complex $A^{-}(s_{1}, s_{2})$ with $(s_{1}, s_{2})\in \bH (L)$ which is introduced by Manolescu and Ozsv\'ath (see Definition 2.3). This complex is a free $\F[U_{1}, U_{2}]$-module. The operations $U_{1}$ and $U_{2}$ are homotopic to each other on each $A^{-}(s_{1}, s_{2})$, cf.\cite{os2} and both have homological degree $-2$.

\begin{definition}
For an oriented link $L=L_{1}\cup L_{2}$ with two components, define $\bH(L)$ to be an affine lattice over $\Z^{2}$,
$$\bH(L)=\bH(L)_{1}\oplus \bH(L)_{2}, \quad \bH(L)_{i}=\Z+\dfrac{lk(L_{1}, L_{2})}{2} (i=1, 2)$$
where $lk(L_{1}, L_{2})$ denotes the linking number of $L_{1}$ and $L_{2}$. 

\end{definition}

By Proposition 1.1, for any $L$-space link $L$ with two components, $H_{\ast}(A^{-}(s_{1}, s_{2}))=\F[[U]]$ where $(s_{1}, s_{2})\in \bH$. Let $-2h(s_{1}, s_{2})$ denote the homological grading of the unique generator in $H_{\ast}(A^{-}(s_{1}, s_{2}))$. The function $h(s_{1}, s_{2})$ is the $HFL$-weight function of an $L$-space link defined in \cite{GN}. In this article, we will call it $h$-function. On each $A^{-}(s_{1}, s_{2})$, the operations $U_{1}$ and $U_{2}$ are homotopic, and we denote them by $U$. 

\begin{lemma}\cite[Lemma 2.2.3]{GN}
\label{2.4}
Let $\mathbf{e_{1}}=(1, 0)$ and $\mathbf{e_{2}}=(0, 1)$. For any  $\mathbf{s}=(s_{1}, s_{2})\in \bH$, there exist  inclusions $ j: A^{-}(s_{1}, s_{2})\hookrightarrow A^{-}(\mathbf{s}+\mathbf{e_{i}})$ for $i=1$ or $i=2$ which induce injections on homology as follows.
$$j_{\ast}: H_{\ast}(A^{-}(s_{1}, s_{2}))\rightarrow H_{\ast}( A^{-}(\mathbf{s}+\mathbf{e_{i}}))$$
where $j_{\ast}=U_{i}^{\delta (i)}$ and $\delta(i)=0$ or $1$.
\end{lemma}

\begin{remark}
The actions $U_{i}$ induce maps $U_{i}: A^{-}(\mathbf{s}+\mathbf{e_{i}})\rightarrow A^{-}(\mathbf{s})$ for $i=1$ and $i=2$. The actions also induce maps on homology. By Proposition 1.1, $H_{\ast}(A^{-}(\mathbf{s}))\cong \F [[U]]$ for any $\mathbf{s}\in \bH$. Assume that $a, b$ are unique generators of $H_{\ast}(A^{-}(\mathbf{s}))$ and $H_{\ast}(A^{-}(\mathbf{s}+\mathbf{e_{i}}))$. Then $j_{\ast}(a)=U^{\delta(i)}b$ and $U_{i}(b)=U^{1-\delta(i)}a$
\end{remark}

\begin{corollary}
\label{coro1}
For any $L$-space link with two components and $\mathbf{s}\in \bH$, $h(\mathbf{s})=h(\mathbf{s}+\mathbf{e_{i}})$ or $h(\mathbf{s})=h(\mathbf{s}+\mathbf{e_{i}})+1$ where $i=1$ or $2$ and $\mathbf{e_{1}}=(1, 0)$ and $\mathbf{e_{2}}=(0, 1)$. 
\end{corollary}

\begin{proof}
By the lemma, we have $-2h(\mathbf{s})=-2h(\mathbf{s}+\mathbf{e_{i}})-2\delta(i)$ where  $\delta (i)=0$ or $1$. So $h(\mathbf{s})=h(\mathbf{s}+\mathbf{e_{i}})$ or $h(\mathbf{s})=h(\mathbf{s}+\mathbf{e_{i}})+1$. 
\end{proof}

Next, we will revisit Yajing Liu's work about  how to use the $h$-function to compute $HFL^{-}(L)$ for any $L$-space link $L=L_{1}\cup L_{2}$ with two components \cite{Liu}. 

\begin{lemma} \cite[Lemma 2.2.9]{GN}
For any  $(s_{1}, s_{2})\in \bH$, the chain complex $CFL^{-}(s_{1}, s_{2})$ of the $L$-space link $L=L_{1}\cup L_{2}$ is quasi-isomorphic to the following ``iterated cone" complex:
\[\mathrm{CFL}^{-}(s_{1}, s_{2}) :\simeq \left[
\begin{tikzcd}
\mathrm{A}^{-}(s_{1}-1, s_{2}) \arrow{r}{i_{1}}  & \mathrm{A}^{-}(s_{1}, s_{2})  \\%
\mathrm{A}^{-}(s_{1}-1, s_{2}-1) \arrow{r}{i_{1}} \arrow[swap]{u}{i_{2}} & \mathrm{A}^{-}(s_{1}, s_{2}-1) \arrow[swap]{u}{i_{2}}
 \end{tikzcd}
\right] \]

where $i_{1}$ and $i_{2}$ are the inclusion maps in Lemma 2.4. 
\end{lemma}

Let $d$ denote the differential in the generalized Floer complex $A^{-}(s_{1}, s_{2})$ and $i=i_{2}-i_{1}$. The above ``iterated cone" complex has two differentials $d$ and $i$. The differential $d$ acts in vertices of the cube and $i$ acts in the edges. Define the cube grading $|K|$ of the upper-right corner of cube to be $0$. The differential $d$ decreases the homological grading by 1 and preserves the cube grading. The differential $i$ preserves the homological grading and decreases the cube grading by 1. The total grading is defined as the sum of homological grading and the cube grading.  Let $D=d+i$ and $K(s_{1}, s_{2})$ denote the above ``iterated cone" complex. There exists a spectral sequence whose $E^{\infty}$ page is the homology of the $K(s_{1}, s_{2})$ under the differential $D=d+i$. 

\begin{theorem}\cite[Theorem 2.2.10]{GN}
Let $L=L_{1}\cup L_{2}$ be an $L$-space link with two components. For any  $(s_{1}, s_{2})\in \bH$, there exists a spectral sequence which converges to $HFL^{-}(s_{1}, s_{2})$ and collapses at its $E^{2}$-page. Its $E^{2}$-page is isomorphic to  $H_{\ast}(H_{\ast}(A^{-}(s_{1}, s_{2}), d), i)$.
\end{theorem}

Thus $HFL^{-}(s_{1}, s_{2})$ is isomorphic to $H_{\ast}(H_{\ast}(A^{-}(s_{1}, s_{2}), d), i)$ in the spectral sequence of the complex $K(s_{1}, s_{2})$. By Proposition 1.1, for any $(s_{1}, s_{2})\in \bH$, $H_{\ast}(A^{-}(s_{1}, s_{2}), d)\cong \F[[U]][-2h(s_{1}, s_{2})]$ where $-2h(s_{1}, s_{2})$ is the homological grading of the unique generator in $H_{\ast}(A^{-}(s_{1}, s_{2}), d)$ and $U_{1}, U_{2}$ act as $U$, homotopic to each other on $A^{-}(s_{1}, s_{2})$, cf. \cite{os2}. To compute $HFL^{-}(s_{1}, s_{2})$, we just need to compute the homology of the  mapping cone of the inclusion map $i$:
\[ \begin{tikzcd}
\mathbb{F}[[U]][-2h(s_{1}-1, s_{2})][b] \arrow{r}{i_{1}}  & \mathbb{F}[[U]][-2h(s_{1}, s_{2})][a] \\%
\mathbb{F}[[U]][-2h(s_{1}-1, s_{2}-1)][c] \arrow{r}{i_{1}} \arrow[swap]{u}{i_{2}} & \mathbb{F}[[U]][-2h(s_{1}, s_{2}-1)][d] \arrow[swap]{u}{i_{2}}
\end{tikzcd}
\]
where $a, b, c, d$ denote the unique generators in $\mathbb{F}[[U]][-2h(s_{1}, s_{2})],\mathbb{F}[[U]][-2h(s_{1}-1, s_{2})],  \mathbb{F}[[U]][-2h(s_{1}-1, s_{2}-1)]$ and $\mathbb{F}[[U]][-2h(s_{1}, s_{2}-1)]$ respectively. Let $h=h(s_{1}, s_{2})$. By Corollary \ref{coro1}, there are 6 cases for the $h$-function corresponding to above mapping cone. 

\begin{figure}[H]
\begin{picture}(100,130)(140,40)
\put(5,60){\framebox(85,40)}
\put(75,68){\makebox(0,0){$h$}}

\put(75,90){\makebox(0,0){$h$}}

\put(25,68){\makebox(0,0){$h+1$}}

\put(20,90){\makebox(0,0){$h$}}

\put(50,50){\makebox(0,0)[I]{Case 4}}

\put(5,120){\framebox(85,40)}
\put(65,128){\makebox(0,0){$h$}}

\put(65,150){\makebox(0,0){$h$}}

\put(25,128){\makebox(0,0){$h$}}

\put(25,150){\makebox(0,0){$h$}}

\put(50, 110){\makebox(0,0)[I]{Case 1}}

\put(125,120){\framebox(85,40)}
\put(195,128){\makebox(0,0){$h+1$}}

\put(185,150){\makebox(0,0){$h$}}

\put(145, 128){\makebox(0,0){$h+1$}}

\put(150, 150){\makebox(0,0){$h$}}

\put(175, 110){\makebox(0,0)[I]{Case 2}}

\put(125,60){\framebox(85,40)}
\put(195,68){\makebox(0,0){$h+1$}}

\put(195,90){\makebox(0,0){$h$}}

\put(145,68){\makebox(0,0){$h+1$}}

\put(145,90){\makebox(0,0){$h+1$}}

\put(175,50){\makebox(0,0)[I]{Case 5}}

\put(245,120){\framebox(85,40)}
\put(310,128){\makebox(0,0){$h$}}

\put(310,150){\makebox(0,0){$h$}}

\put(265,128){\makebox(0,0){$h+1$}}

\put(265,150){\makebox(0,0){$h+1$}}

\put(290, 110){\makebox(0,0)[I]{Case 3}}

\put(245,60){\framebox(85,40)}
\put(315,68){\makebox(0,0){$h+1$}}

\put(310,90){\makebox(0,0){$h$}}

\put(265,68){\makebox(0,0){$h+2$}}

\put(265,90){\makebox(0,0){$h+1$}}

\put(290, 50){\makebox(0,0)[I]{Case 6}}
\end{picture}
\caption{Possible local behaviours of $h$-function}
\end{figure}
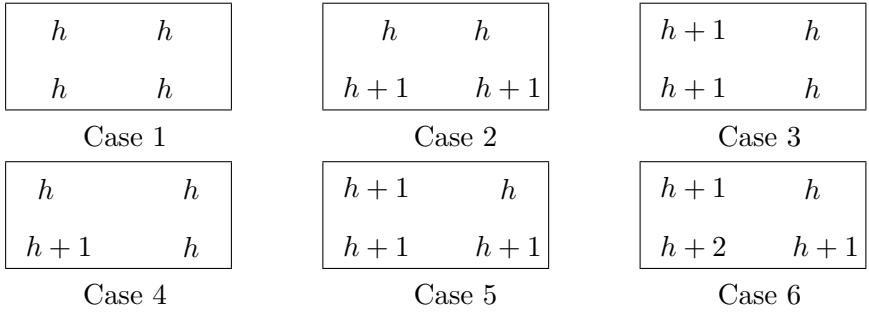

According to the $h$-function in Figure 1, we can compute the corresponding $HFL^{-}(s_{1}, s_{2})$ in each case. 

\textbf{Case 1:} $i(b)=a, i(c)=b-d, i(d)=a$ and $i(a)=0$, so $HFL^{-}(s_{1}, s_{2})=0$. 

\textbf{Case 2:} $i(b)=a, i(c)=Ub-d, i(d)=Ua$ and $i(a)=0$, so $HFL^{-}(s_{1}, s_{2})=0$. 

\textbf{Case 3:} $i(b)=Ua, i(c)=b-Ud, i(d)=a$ and $i(a)=0$, so $HFL^{-}(s_{1}, s_{2})=0$.

\textbf{Case 4:} $i(b)=a, i(c)=Ub-Ud, i(d)=a$ and $i(a)=0$, so $HFL^{-}(s_{1}, s_{2})=\langle b-d \rangle$. Both $b$ and $d$ have homological grading $-2h$ and cube grading $1$. The total grading of $b-d$ is $-2h+1$. Thus  $HFL^{-}(s_{1}, s_{2})=\F[-2h+1]$. 

\textbf{Case 5:} $i(b)=Ua, i(c)=b-d, i(d)=Ua$ and $i(a)=0$, so $HFL^{-}=\langle a \rangle$ with total grading $-2h$ . Thus $HFL^{-}(s_{1}, s_{2})=\F[-2h]$.

\textbf{Case 6:} $i(b)=Ua, i(c)=Ub-Ud, i(d)=Ua$ and $i(a)=0$, so $HFL^{-}(s_{1}, s_{2})=\langle a, b-d \rangle$. Here $a$ has total grading $-2h$ and $b-d$ has total grading $-2(h+1)+1=-2h-1$. Thus $HFL^{-}(s_{1}, s_{2})=\F[-2h]\oplus \F[-2h-1]$.

Moreover, we can also determine the Euler characteristics $\chi(HFL^{-}(s_{1}, s_{2}))$ in these six cases. In Case 1, Case 2, Case 3 and Case 6,  $\chi(HFL^{-}(s_{1}, s_{2}))=0$. In Case 4,  $\chi(HFL^{-}(s_{1}, s_{2}))=-1$ and in Case 5,  $\chi(HFL^{-}(s_{1}, s_{2}))=1$. Thus for any $L$-space link $L=L_{1}\cup L_{2}$ with two components, once the $h$-function is determined, we can compute $HFL^{-}(s_{1}, s_{2})$ for any  $(s_{1}, s_{2})\in \bH$. 

\begin{corollary}
For any $L$-space link with two components and $(s_{1}, s_{2})\in \bH$, $HFL^{-}(s_{1}, s_{2})$ is spanned by  $a$ or $b-d$ or both. $a$ has even grading and $b-d$ has odd grading. 
\end{corollary}

\subsection{Alexander Polynomials of $L$-Space Links}

In this section, we mainly introduce Yajing Liu's work \cite{Liu} about how to determine the $h$-function of any $L$-space link $L=L_{1}\cup L_{2}$  with two components $L_{1}, L_{2}$ by the multi-variable Alexander polynomial $\Delta_{L}(t_{1}, t_{2})$, the Alexander polynomials $\Delta_{L_{1}}(t)$ and $\Delta_{L_{2}}(t)$. Recall that for any $L$-space link  $L=L_{1}\cup L_{2}$ with two components, we have : 
$$\Delta_{L}(t_{1}, t_{2})\doteq \sum\limits_{(s_{1}, s_{2})\in \bH} \chi(HFL^{-}(s_{1}, s_{2})) t_{1}^{s_{1}} t_{2}^{s_{2}} $$

\begin{equation}
\Delta_{L}(t, 1) \doteq \dfrac{1-t^{lk}}{1-t} \Delta_{L_{1}}(t)
\end{equation}
where $f \doteq g$ means that $f$ and $g$ differ by multiplication by units. For any  $L$-space link $L=L_{1}\cup L_{2}$ with two components, Yajing Liu also defined normalization of its Alexander polynomial  \cite{Liu}. 

\begin{definition}(\cite[Definition 5.12]{Liu})
Let the symmetrized Alexander polynomial of $L$ be $\Delta_{L}(x_{1}, x_{2})$ in the form of 
$$\Delta_{L}(t_{1}, t_{2})=\sum\limits_{i, j} a_{i,j}^{L} \cdot t_{1}^{i} \cdot t_{2}^{j}$$
where $t_{i}$ corresponds to the link component $L_{i}$ for $i=1$ and $i=2$. Let the symmetrized Alexander polynomials of $L_{1}$ and $L_{2}$ be $\Delta_{L_{1}}(t), \Delta_{L_{2}}(t)$ in the form of 
$$\dfrac{t}{t-1} \Delta_{L_{1}}(t)=\sum\limits_{k\in \Z} a_{k}^{L_{1}} \cdot t^{k}, \quad
\dfrac{t}{t-1} \Delta_{L_{2}}(t)=\sum\limits_{k\in \Z} a_{k}^{L_{2}}\cdot t^{k} $$
Let $(i_{0}, j_{0})$ be such that 
$$j_{0}=\textup{max} \left\lbrace j \in \Z +\dfrac{lk-1}{2} \mid a_{i,j}^{L}\neq 0 \right\rbrace \quad \textup{and} \quad  i_{0}=\textup{max} \left\lbrace i\in \Z +\dfrac{lk-1}{2} \mid a_{i, j_{0}}^{L}\neq  0 \right\rbrace$$  
Then these Alexander polynomials are called $\mathbf{normalized}$, if 

(1) the leading coefficient of $\Delta_{L_{i}}(t)$ is $1$ for both $i=1, 2$.

(2) if $a_{j_{0}-lk/2+1/2}^{L_{2}}=1$, then $a_{i_{0}, j_{0}}^{L}=1$; while if $a^{L_{2}}_{j_{0}-lk/2+1/2}=0$, then $a_{i_{0}, j_{0}}^{L}=-1$ where $lk$ is the linking number of $L_{1}$ and $L_{2}$. 
\end{definition}

For the normalized Alexander polynomials of the 2-component $L$-space link $L=L_{1}\cup L_{2}$, $\chi(HFL^{-})(s_{1}, s_{2})=a^{L}_{s_{1}-1/2, s_{2}-1/2}$ and $\chi(HFK^{-}(L_{i}, s))=a^{L_{i}}_{s}$ for $i=1, 2$ \cite{Liu}. Moreover, Yajing Liu gives the following formulas to determine the $h$-functions in \cite[equation (5.8)]{Liu}:
\begin{equation}
\label{h-function 1}
h(s_{1}, s_{2}-1)-h(s_{1}, s_{2})=a^{L_{2}}_{s_{2}-lk/2} - \sum\limits_{i=1}^{\infty} a^{L}_{s_{1}+i-1/2, s_{2}-1/2}=0 \quad\textup{or}\quad 1
\end{equation}

Similary
\begin{equation}
\label{h-fucntion 2}
h(s_{1}-1, s_{2})-h(s_{1}, s_{2})=a^{L_{1}}_{s_{1}-lk/2} - \sum\limits_{i=1}^{\infty} a^{L}_{s_{1}-1/2, s_{2}+i-1/2}=0 \quad\textup{or}\quad 1
\end{equation}

When $s_{1}\rightarrow +\infty$ or $s_{2}\rightarrow +\infty$
\begin{equation}
\label{h-function 3}
h(+\infty, s_{2})=h_{2}(s_{2}-lk/2)
\quad h(s_{1}, +\infty)=h_{1}(s_{1}-lk/2)
\end{equation}

\begin{equation}
\label{h-function 3}
h_{1}(s-1)-h_{1}(s)=a_{s}^{L_{1}}  
\quad h_{2}(s-1)-h_{2}(s)=a_{s}^{L_{2}}
\end{equation}
where $h_{1}(s_{1}-lk/2), h_{2}(s_{2}-lk/2)$ are the corresponding $h$-functions for link components $L_{1}$ and $L_{2}$ respectively and $s\in \Z$. For $s$ sufficiently large, $h_{1}(s)=h_{2}(s)=0$. By using the above formulas, we can compute $HFL^{-}(s_{1}, s_{2})$ for any $L$-space link $L=L_{1}\cup L_{2}$ with two components and $(s_{1}, s_{2})\in \bH$.

\begin{remark}
The link components $L_{1}$ and $L_{2}$ of the $2$-component $L$-space link are both $L$-space knots \cite[Lemma 1.10]{Liu}.

\end{remark}

\begin{corollary} \cite{ ND, GN, Liu}
For any $L$-space link $L=L_{1}\cup L_{2}$ with two components, $HFL^{-}(L)$ is determined by the Alexander polynomials $\Delta_{L}(t_{1}, t_{2}), \Delta_{L_{1}}(t)$ and $\Delta_{L_{2}}(t)$.

\end{corollary} 

\section{Computation of $\widehat{HFL}$ for $L$-space links with two components }
\subsection{The spectral sequence corresponding to $\widehat{HFL}$}

In Section 2, we proved that for any $L$-space link $L=L_{1}\cup L_{2}$ with  $(s_{1}, s_{2})\in \bH$, $HFL^{-}(s_{1}, s_{2})$ is determined by the $h$-function. Now we are going to prove Theorem 1.2 that the Heegaard Floer link homology $\widehat{HFL}(s_{1}, s_{2})$ is also determined by the $h$-function.

Let $\mathfrak{C}(s_{1}, s_{2})=CFL^{-}(s_{1}+1, s_{2}+1)\oplus CFL^{-}(s_{1}+1, s_{2})\oplus CFL^{-}(s_{1}, s_{2}+1)\oplus CFL^{-}(s_{1}, s_{2})$.  For any $(s_{1}, s_{2})\in \bH$, we have two operators $U_{1}: CFL^{-}(s_{1}, s_{2})\rightarrow CFL^{-}(s_{1}-1, s_{2})$ and $U_{2}: CFL^{-}(s_{1}, s_{2})\rightarrow CFL^{-}(s_{1}, s_{2}-1)$. The action of $U_{1}$ (or $U_{2}$) is defined by $h$-function (See Lemma \ref{U-action}).  Let $D=d+d_{1}$ where $d$ is the differential in chain complex $CFL^{-}(s_{1}, s_{2})$ with any $(s_{1}, s_{2})\in \bH$ and $d_{1}=U_{1}-U_{2}$.  Then we get the  ``iterated cone" complex $(\mathfrak{C}(s_{1}, s_{2}), d+d_{1})$ in the following form:

\[ \left[
\begin{tikzcd}
CFL^{-}(s_{1}, s_{2}+1) \arrow[swap]{d}{U_{2}}  & CFL^{-}(s_{1}+1, s_{2}+1) \arrow{l}{U_{1}} \arrow[swap]{d}{U_{2}}  \\%
CFL^{-}(s_{1}, s_{2})  & CFL^{-1}(s_{1}+1, s_{2}) \arrow{l}{U_{1}}
 \end{tikzcd}
\right] \]

\begin{lemma}
Suppose that $L=L_{1}\cup L_{2}$ is an $L$-space link with two components $L_{1}$ and $L_{2}$. Let $\widehat{CFL}(s_{1}, s_{2})$ denote the chain complex of hat-version of Heegaard Floer link homology of $L$ with $(s_{1}, s_{2})\in \bH$. Then $\widehat{CFL}(s_{1}, s_{2})$ is quasi-isomorphic to the ``iterated cone" complex $(\mathfrak{C}(s_{1}, s_{2}), d+d_{1})$. 
\end{lemma}

\begin{proof}
For this $L$-space link $L=L_{1}\cup L_{2}$ with two components, we can write $\widehat{CFL}(s_{1}, s_{2})$ as:
$$\dfrac{CFL^{-}(s_{1}, s_{2})/ U_{1}(CFL^{-}(s_{1}+1, s_{2}))}{U_{2}(CFL^{-}(s_{1}, s_{2}+1)/U_{1}(CFL^{-}(s_{1}+1, s_{2}+1)))}$$
The quotient $CFL^{-}(s_{1}, s_{2})/ U_{1}(CFL^{-}(s_{1}+1, s_{2})$ can be realized as the  cone of the map $U_{1}: CFL^{-}(s_{1}+1, s_{2})\rightarrow CFL^{-}(s_{1}, s_{2})$ and similarly the quotient $CFL^{-}(s_{1}, s_{2}+1)/ U_{1}(CFL^{-}(s_{1}+1, s_{2}+1))$ can be realized as the cone of the map $U_{1}:CFL^{-}(s_{1}+1, s_{2}+1)\rightarrow CFL^{-}(s_{1}, s_{2}+1)$. Thus $\widehat{CFL}(s_{1}, s_{2})$ can be realized as a cone of the natural map induced by $U_{2}$ between these two  cones. 
\end{proof}

\begin{theorem}
\label{thm3}
Let $L=L_{1}\cup L_{2}$ be an $L$-space link with two components. For any $(s_{1}, s_{2})\in \bH$, there exists a spectral sequence with the following properties:

(a) Its $E^{2}$-page is isomorphic (as graded $\F$-module) to $H_{\ast}(H_{\ast}(\mathfrak{C}(s_{1}, s_{2}), d), d_{1})$.

(b) Its $E^{\infty}$-page is isomorphic (as graded $\F$-module) to $\widehat{HFL}(s_{1}, s_{2})$.

(c) The spectral sequence collapses at $E^{3}$. 
\end{theorem}

\begin{proof}
For the ``iterated cone" complex $\mathfrak{C}(s_{1}, s_{2})$ we mentioned above, it is doubly graded. One is the homological grading $\nu$ in the chain complex $CFL^{-}(s_{1}, s_{2})$ with $(s_{1}, s_{2})\in \bH$. We define \emph{cube grading} $|C|$ in the cube of the `iterated cone" complex $\mathfrak{C}(s_{1}, s_{2})$. Fix $(s_{1}, s_{2})\in \bH$. The cube grading is defined as $(s_{1}+s_{2})-(v_{1}+v_{2})$ where $(v_{1}, v_{2})\in \bH$. It is equivalent to saying the the cube grading of the lower-left corner of the cube is $0$ and $U_{1}$ (or $U_{2}$) increases the cube grading by 1. 

For this doubly-graded complex $\mathfrak{C}(s_{1}, s_{2})$ with two (anti)commuting differentials $d$ and $d_{1}$, there exists a spectral sequence whose $E^{1}$-page is $H_{\ast}(\mathfrak{C}(s_{1}, s_{2}), d)$  and converges to $H_{\ast}(\mathfrak{C}(s_{1}, s_{2}), d+d_{1})$. By Lemma 3.1, its $E^{\infty}$-page is  isomorphic to $\widehat{HFL}(s_{1}, s_{2})$. Its $E^{1}$-page is $H_{\ast}(\mathfrak{C}(s_{1}, s_{2}), d)$ which can be written as $HFL^{-}(s_{1}+1, s_{2}+1)\oplus HFL^{-}(s_{1}+1, s_{2})\oplus HFL^{-}(s_{1}, s_{2}+1)\oplus HFL^{-}(s_{1}, s_{2})$. Its $E^{2}$-page is  $H_{\ast}(H_{\ast}(\mathfrak{C}(s_{1}, s_{2}), d), d_{1})$. The differential $d_{0}=d$ in the spectral sequence preserves the cube degree $|C|$ and decreases the homological degree $\nu$ by 1. The differential $d_{1}$ in $E^{1}$-page increases the cube degree by $1$ and decreases the homological degree $\nu$ by 2. For any nonnegative integer $k$, the differential $d_{k}$ increases the cube degree by $k$, decreases the homological degree $\nu$ by $k+1$. The total homological degree is $\nu+|C|$. By grading reason, the cube grading is less than or equal to $2$. Thus for integer $k>2$, $d_{k}=0$ and this spectral sequence collapses at $E^{3}$. 
\end{proof}

By Theorem 3.2, $\widehat{HFL}(s_{1}, s_{2})\cong E^{3}$. Then we can compute $\widehat{HFL}(s_{1}, s_{2})$ by computing $E^{3}$-page of the above spectral sequence. The following lemma describes the action of $U_{1}$ (or $U_{2}$) on the $E^{1}$-page of the spectral sequence in Theorem \ref{thm3}. This gives the action of $d_{1}$ on $E^{1}$-page. 

\begin{lemma}
\label{U-action}
Consider the map $U_{1}: HFL^{-}(s_{1}+1, s_{2}+1)\rightarrow HFL^{-}(s_{1}, s_{2}+1)$. Let $\alpha$  be a generator of $HFL^{-}(s_{1}+1, s_{2}+1)$ with homological grading $x$. If there exists a generator $\beta$ in $HFL^{-}(s_{1}, s_{2}+1)$ with homological grading $x-2$, then $U_{1}(\alpha)=\beta$. 
\end{lemma}

\begin{proof}
Let $a_{1}, b_{1}, c_{1}, d_{1}$ denote the unique generators of $H_{\ast}(A^{-}(s_{1}, s_{2}+1))$, $H_{\ast}(A^{-}(s_{1}-1, s_{2}+1))$, $H_{\ast}(A^{-}(s_{1}-1, s_{2}))$ and $H_{\ast}(A^{-}(s_{1}, s_{2}))$ respectively in Figure 2.  Let $a, b,c, d$ denote the unique generators of $H_{\ast}(A^{-}(s_{1}+1, s_{2}+1))$, $H_{\ast}(A^{-}(s_{1}, s_{2}+1))$, $H_{\ast}(A^{-}(s_{1}, s_{2}))$ and  $H_{\ast}(A^{-}(s_{1}+1, s_{2}))$ respectively in Figure 2. Here $a_{1}$ and $b$ have different cube gradings as the generators of $H_{\ast}(A^{-}(s_{1}, s_{2}+1))$ and $d_{1}, c$ have different cube gradings as the generators of $H_{\ast}(A^{-}(s_{1}, s_{2}))$. By the computation of $HFL^{-}$ in Section 2.1, $h(s_{1}, s_{2}+1)=h(s_{1}+1, s_{2})$ once $HFL^{-}(s_{1}+1, s_{2}+1)$ is nonempty. By the same reason, $h(s_{1}-1, s_{2}+1)=h(s_{1}, s_{2})$ since $HFL^{-}(s_{1}, s_{2})$ is also nonempty. Assume the generator $\alpha=b-d$ with total homological grading $-2h(s_{1}, s_{2}+1)+1$. The generator $a_{1}$ has total homological grading $-2h(s_{1}, s_{2}+1)$ and $b_{1}-d_{1}$ has total homological grading $-2h(s_{1}-1, s_{2}+1)+1$. By the assumption of this lemma, the total homological grading of $\beta$ is $-2h(s_{1}, s_{2}+1)-1$, so $\beta$ can only be $b_{1}-d_{1}$ and $h(s_{1}-1, s_{2}+1)=h(s_{1}, s_{2}+1)+1$. Now consider the map $U_{1}: H_{\ast}(A^{-}(s_{1}, s_{2}+1))\rightarrow H_{\ast}(A^{-}(s_{1}-1, s_{2}+1))$ where $H_{\ast}(A^{-}(s_{1}, s_{2}+1))=\langle b \rangle$ and $H_{\ast}(A^{-}(s_{1}-1, s_{2}+1))=\langle b_{1} \rangle$. Since $U_{1}$ has homological degree $-2$, $U_{1}(d)=d_{1}$ by Lemma \ref{2.4} and Remark 2.5. Similarly, $U_{1}(c)=c_{1}$. Then $U_{1}(\alpha)=U_{1}(b-d)=b_{1}-d_{1}=\beta$. If $\alpha=a$, then $\beta=a_{1}$ and we use the similar argument above to prove $U_{1}(\alpha)=\beta$.

\begin{figure}[H]
\begin{picture}(100,70)(150,5)
\put(1,10){\framebox(190,50)}
\put(134,22){\makebox(0,0){$A^{-}(s_{1}, s_{2})$}}
\put(165,22){\makebox(0,0){$[d_{1}]$}}
\put(140,50){\makebox(0,0){$A^{-}(s_{1}, s_{2}+1)$}}
\put(181,50){\makebox(0,0){$[a_{1}]$}}
\put(38,22){\makebox(0,0){$A^{-}(s_{1}-1, s_{2})$}}
\put(78,22){\makebox(0,0){$[c_{1}]$}}
\put(43,50){\makebox(0,0){$A^{-}(s_{1}-1, s_{2}+1)$}}
\put(93,50){\makebox(0,0){$[b_{1}]$}}
\put(85, 1){\makebox(0,0)[I]{$HFL^{-}(s_{1}, s_{2}+1)$}}

\put(210,10){\framebox(190,50)}
\put(347,22){\makebox(0,0){$A^{-}(s_{1}+1, s_{2})$}}
\put(385,22){\makebox(0,0){$[d]$}}
\put(342,50){\makebox(0,0){$A^{-}(s_{1}+1, s_{2}+1)$}}
\put(390,50){\makebox(0,0){$[a]$}}
\put(248, 22){\makebox(0,0){$A^{-}(s_{1}, s_{2})$}}
\put(277, 22){\makebox(0,0){$[c]$}}
\put(244, 50){\makebox(0,0){$A^{-}(s_{1}, s_{2}+1)$}}
\put(282,50){\makebox(0,0){$[b]$}}
\put(300, 1){\makebox(0,0)[I]{$HFL^{-}(s_{1}+1, s_{2}+1)$}}

\end{picture}
\caption{}
\end{figure}

\end{proof}

\begin{remark}
The map $U_{2}: HFL^{-}(s_{1}+1, s_{2}+1) \rightarrow HFL^{-}(s_{1}+1, s_{2} )$ can be described similarly to Lemma \ref{U-action}.
\end{remark}

For the action of $d_{2}$ on $E^{2}$-page, if it is nontrivial, we use the symmetry of Heegaard Floer link homology. 

\begin{lemma}\cite[Equation 5]{os3}
\label{symmetry}
For an oriented $L$-space link $L=L_{1}\cup L_{2}$ with two components and  $\mathbf{s}=(s_{1}, s_{2})\in \bH$, there exists a relatively graded isomorphism 
$$\widehat{HFL}(L, \mathbf{s})\cong \widehat{HFL}(L, - \mathbf{s}) $$
\end{lemma}

\begin{remark}
In particular, the $h$-functions satisfy that $h(-\mathbf{s})=h(\mathbf{s})+|\mathbf{s}|$ \cite[Lemma 5.5]{Liu} where $|\mathbf{s}|=s_{1}+s_{2}.$ 
\end{remark}

\subsection{Proof of the main Theorem }
In this subsection, we are going to give the proof of Theorem \ref{thm 2}  and show that $4$ is an upper bound for the rank of link Floer homology $\widehat{HFL}(s_{1}, s_{2})$ for any $2$-component $L$-space link and $(s_{1}, s_{2})\in \bH$. Example 3.8 gives a $2$-component $L$-space link where the rank of $\widehat{HFL}(s_{1}, s_{2})$ ranges from $0$ to $4$. \\

\noindent
{\bf Proof of Theorem \ref{thm 2}:}
Assume that $h=h(s_{1}+1, s_{2}+1)$. If $d_{2}=0$, then the spectral sequence in Theorem \ref{thm3} collapses at $E^{2}$-page, then we can use the computation of $HFL^{-}$ in Section 2.1 and Lemma \ref{U-action}  to compute $\widehat{HFL}(s_{1}, s_{2})$. For example, suppose that  the  $h$-function corresponding to  $\widehat{HFL}(s_{1}, s_{2})$ is the following:
\begin{figure}[ht]
\begin{picture}(110,65)(30,1)
\put(1,10){\framebox(110, 60)}
\put(90, 20){\makebox(0,0){$h+1$}}
\put(55, 20){\makebox(0,0){$h+1$}}
\put(20, 20){\makebox(0,0){$h+2$}}
\put(90, 40){\makebox(0,0){$h$}}
\put(55, 40){\makebox(0,0){$h+1$}}
\put(20, 40){\makebox(0,0){$h+1$}}
\put(90, 60){\makebox(0,0){$h$}}
\put(55, 60){\makebox(0,0){$h$}}
\put(20, 60){\makebox(0,0){$h+1$}}
\end{picture}
\end{figure}

We obtain that  the $E^{2}$-page of the spectral sequence is:
\[\widehat{HFL}(s_{1}, s_{2}) :\simeq \left[
\begin{tikzcd}
\mathbb{F}[-2h] \arrow[swap]{d}{U_{2}}  & \mathbb{F}[-2h+1]  \arrow{l}{U_{1}} \arrow[swap]{d}{U_{2}}  \\%
\mathbb{F}[-2h-1]  & \mathbb{F}[-2h] \arrow{l}{U_{1}}
 \end{tikzcd}
\right] \]
Recall that $U_{1}$ and $U_{2}$ both have homological grading $-2$. So $U_{1}=U_{2}=0$. By Theorem \ref{thm3}, we know that $d_{2}$ increases the cube grading by $2$ and decreases the homological grading $\nu$ by $3$, so $d_{2}=0$ by grading reason. Thus $\widehat{HFL}(s_{1}, s_{2})\cong \F[-2h-1]\oplus \F[-2h-1] \oplus \F[-2h-1] \oplus \F[-2h-1]$. Here the cube grading for the generator in $\F[-2h-1]$ is $0$. We can use this method to compute $\widehat{HFL}$ in all the cases where $d_{2}=0$. Now it suffices to consider the cases where $d_{2}$ may be nontrivial.

By grading reason, in order to have nontrivial $d_{2}$, $HFL^{-}(s_{1}+1, s_{2}+1)$ and $HFL^{-}(s_{1}, s_{2})$ are both nonzero and contain a generator in each group such that their homological degree difference is $3$. For nonzero $HFL^{-}(s_{1}+1, s_{2}+1)$, we have the following three possibilities for the corresponding $h$-function :

\begin{figure}[H]
\begin{picture}(100,50)(120,1)
\put(10,10){\framebox(80,50)}
\put(65,22){\makebox(0,0){$h$}}
\put(65,50){\makebox(0,0){$h$}}
\put(30,22){\makebox(0,0){$h+1$}}
\put(30,50){\makebox(0,0){$h$}}
\put(45, 1){\makebox(0,0)[I]{Case 1}}

\put(110,10){\framebox(80,50)}
\put(165,22){\makebox(0,0){$h+1$}}
\put(165,50){\makebox(0,0){$h$}}
\put(130, 22){\makebox(0,0){$h+1$}}
\put(130, 50){\makebox(0,0){$h+1$}}
\put(145, 1){\makebox(0,0)[I]{Case 2}}

\put(210,10){\framebox(80,50)}
\put(270,22){\makebox(0,0){$h+1$}}
\put(270,50){\makebox(0,0){$h$}}
\put(235,22){\makebox(0,0){$h+2$}}
\put(235,50){\makebox(0,0){$h+1$}}
\put(245, 1){\makebox(0,0)[I]{Case 3}}
\end{picture}
\end{figure}

In Case 1, $HFL^{-}(s_{1}+1, s_{2}+1)=\F[-2h+1]$. In order to have nontrivial $d_{2}$,  $HFL^{-}(s_{1}, s_{2})$ must contain one generator with homological grading $-2h-2$ by grading reason. So the $h$-function corresponding to $HFL^{-}(s_{1}, s_{2})$ can only have the pattern as in Case 2 or Case 3. Once the $h$-function is determined for $HFL^{-}(s_{1}, s_{2})$, the $h$-functions for $HFL^{-}(s_{1}, s_{2}+1)$ and $HFL^{-}(s_{1}+1, s_{2})$ are also determined by  Corollary \ref{coro1}. Thus  there are two possibilities for the $h$-function corresponding to $\widehat{HFL}(s_{1}, s_{2})$ where $d_{2}$ may be nontrivial:

\begin{figure}[H]
\begin{picture}(100,70)(120,0)
\put(10,10){\framebox(110,60)}
\put(100, 20){\makebox(0,0){$h+1$}}
\put(65, 20){\makebox(0,0){$h+2$}}
\put(30, 20){\makebox(0,0){$h+2$}}
\put(100, 40){\makebox(0,0){$h$}}
\put(65, 40){\makebox(0,0){$h+1$}}
\put(30, 40){\makebox(0,0){$h+2$}}
\put(100, 60){\makebox(0,0){$h$}}
\put(65, 60){\makebox(0,0){$h$}}
\put(30, 60){\makebox(0,0){$h+1$}}
\put(60, 1){\makebox(0,0)[I]{Case $(1a)$}}

\put(180,10){\framebox(110,60)}
\put(200, 20){\makebox(0,0){$h+3$}}
\put(270, 20){\makebox(0,0){$h+1$}}
\put(235, 20){\makebox(0,0){$h+2$}}
\put(200, 40){\makebox(0,0){$h+2$}}
\put(270, 40){\makebox(0,0){$h$}}
\put(235, 40){\makebox(0,0){$h+1$}}
\put(200, 60){\makebox(0,0){$h+1$}}
\put(270, 60){\makebox(0,0){$h$}}
\put(235, 60){\makebox(0,0){$h$}}
\put(230, 1){\makebox(0,0)[I]{Case $(1b)$}}
\end{picture}
\end{figure}

In both cases, $HFL^{-}(s_{1}+1, s_{2}+1)=\F[-2h+1], HFL^{-}(s_{1}, s_{2}+1)=\F[-2h]\oplus \F[-2h-1]$ and $HFL^{-}(s_{1}+1, s_{2})=\F[-2h]\oplus \F[-2h-1]$. By Lemma \ref{U-action}, $U_{1}a=b$ and $U_{2}a=c$ where $a$ is the unique generator in $HFL^{-}(s_{1}+1, s_{2}+1)$, and $b, c$ are generators with homological grading $-2h-1$ in $ HFL^{-}(s_{1}, s_{2}+1)$ and $HFL^{-}(s_{1}+1, s_{2})$ respectively. So the image of $a$ under the differential $d_{1}$ is nonzero and $a$ does not survive in $E^{2}$-page of the spectral sequence. Thus $d_{2}$ is trivial in both Case $(1a)$ and Case $(1b)$. 

In Case 2, $HFL^{-}(s_{1}+1, s_{2}+1)=\F[-2h]$. In order to have nontrivial $d_{2}$, $HFL^{-}(s_{1}, s_{2})$ must contain a generator with homological degree $[-2h-3]$. So the $h$-function of $HFL^{-}(s_{1}, s_{2})$ must have the pattern as in Case 3. Then $HFL^{-}(s_{1}, s_{2})\cong \F[-2h-2]\oplus \F[-2h-3]$. Corresponding to this case, there are four possibilities of the $h$-function for $\widehat{HFL}(s_{1}, s_{2})$:
\begin{figure}[H]
\begin{picture}(100,70)(165,0)
\put(1,10){\framebox(100,60)}
\put(15, 20){\makebox(0,0){$h+3$}}
\put(15, 40){\makebox(0,0){$h+2$}}
\put(15, 60){\makebox(0,0){$h+1$}}
\put(50, 20){\makebox(0,0){$h+2$}}
\put(50, 40){\makebox(0,0){$h+1$}}
\put(50, 60){\makebox(0,0){$h+1$}}
\put(85, 20){\makebox(0,0){$h+1$}}
\put(85, 40){\makebox(0,0){$h+1$}}
\put(85, 60){\makebox(0,0){$h$}}
\put(55, 1){\makebox(0,0)[I]{Case $(2a)$}}

\put(110,10){\framebox(100,60)}
\put(125, 20){\makebox(0,0){$h+3$}}
\put(125, 40){\makebox(0,0){$h+2$}}
\put(125, 60){\makebox(0,0){$h+2$}}
\put(160, 20){\makebox(0,0){$h+2$}}
\put(160, 40){\makebox(0,0){$h+1$}}
\put(160, 60){\makebox(0,0){$h+1$}}
\put(195, 20){\makebox(0,0){$h+2$}}
\put(195, 40){\makebox(0,0){$h+1$}}
\put(195, 60){\makebox(0,0){$h$}}
\put(165, 1){\makebox(0,0)[I]{Case $(2b)$}}

\put(220, 10){\framebox(100, 60)}
\put(235, 20){\makebox(0,0){$h+3$}}
\put(235, 40){\makebox(0,0){$h+2$}}
\put(235, 60){\makebox(0,0){$h+2$}}
\put(270, 20){\makebox(0,0){$h+2$}}
\put(270, 40){\makebox(0,0){$h+1$}}
\put(270, 60){\makebox(0,0){$h+1$}}
\put(305, 20){\makebox(0,0){$h+1$}}
\put(305, 40){\makebox(0,0){$h+1$}}
\put(305, 60){\makebox(0,0){$h$}}
\put(275, 1){\makebox(0,0)[I]{Case $(2c)$}}

\put(330, 10){\framebox(100, 60)}
\put(345, 20){\makebox(0,0){$h+3$}}
\put(345, 40){\makebox(0,0){$h+2$}}
\put(345, 60){\makebox(0,0){$h+1$}}
\put(380, 20){\makebox(0,0){$h+2$}}
\put(380, 40){\makebox(0,0){$h+1$}}
\put(380, 60){\makebox(0,0){$h+1$}}
\put(415, 20){\makebox(0,0){$h+2$}}
\put(415, 40){\makebox(0,0){$h+1$}}
\put(415, 60){\makebox(0,0){$h$}}
\put(385, 1){\makebox(0,0)[I]{Case $(2d)$}}

\end{picture}
\end{figure}

We will use the symmetry of Heegaard Floer link homology to compute $\widehat{HFL}(s_{1}, s_{2})$. Let $h^{\ast}=h(-s_{1}, -s_{2})$. By Remark 3.6, $h(-s_{1}, -s_{2}-1)-h(-s_{1}, -s_{2})=1-(h(s_{1}, s_{2})-h(s_{1}, s_{2}+1))$ and $h(-s_{1}-1, -s_{2})-h(-s_{1}, -s_{2})=1-(h(s_{1}, s_{2})-h(s_{1}+1, s_{2}))$. So the $h$-function for $\widehat{HFL}(-s_{1}, -s_{2})$ corresponding to these four subcases are  :
\begin{figure}[H]
\begin{picture}(100,60)(165,5)
\put(1,10){\framebox(100,60)}
\put(20, 20){\makebox(0,0){$h^{\ast}+1$}}
\put(20, 40){\makebox(0,0){$h^{\ast}+1$}}
\put(20, 60){\makebox(0,0){$h^{\ast}$}}
\put(60, 20){\makebox(0,0){$h^{\ast}+1$}}
\put(60, 40){\makebox(0,0){$h^{\ast}$}}
\put(60, 60){\makebox(0,0){$h^{\ast}$}}
\put(90, 20){\makebox(0,0){$h^{\ast}$}}
\put(90, 40){\makebox(0,0){$h^{\ast}$}}
\put(90, 60){\makebox(0,0){$h^{\ast}$}}
\put(55, 1){\makebox(0,0)[I]{Dual-h $(2a)$}}

\put(110,10){\framebox(100,60)}
\put(125, 20){\makebox(0,0){$h^{\ast}+1$}}
\put(125, 40){\makebox(0,0){$h^{\ast}+1$}}
\put(125, 60){\makebox(0,0){$h^{\ast}+1$}}
\put(160, 20){\makebox(0,0){$h^{\ast}+1$}}
\put(160, 40){\makebox(0,0){$h^{\ast}$}}
\put(160, 60){\makebox(0,0){$h^{\ast}$}}
\put(195, 20){\makebox(0,0){$h^{\ast}+1$}}
\put(195, 40){\makebox(0,0){$h^{\ast}$}}
\put(195, 60){\makebox(0,0){$h^{\ast}$}}
\put(165, 1){\makebox(0,0)[I]{Dual-h $(2b)$}}

\put(220, 10){\framebox(100, 60)}
\put(236, 20){\makebox(0,0){$h^{\ast}+1$}}
\put(236, 40){\makebox(0,0){$h^{\ast}+1$}}
\put(236, 60){\makebox(0,0){$h^{\ast}$}}
\put(271, 20){\makebox(0,0){$h^{\ast}+1$}}
\put(271, 40){\makebox(0,0){$h^{\ast}$}}
\put(271, 60){\makebox(0,0){$h^{\ast}$}}
\put(306, 20){\makebox(0,0){$h^{\ast}+1$}}
\put(306, 40){\makebox(0,0){$h^{\ast}$}}
\put(306, 60){\makebox(0,0){$h^{\ast}$}}
\put(275, 1){\makebox(0,0)[I]{Dual-h $(2c)$}}

\put(330, 10){\framebox(100, 60)}
\put(347, 20){\makebox(0,0){$h^{\ast}+1$}}
\put(347, 40){\makebox(0,0){$h^{\ast}+1$}}
\put(347, 60){\makebox(0,0){$h^{\ast}+1$}}
\put(383, 20){\makebox(0,0){$h^{\ast}+1$}}
\put(383, 40){\makebox(0,0){$h^{\ast}$}}
\put(383, 60){\makebox(0,0){$h^{\ast}$}}
\put(418, 20){\makebox(0,0){$h^{\ast}$}}
\put(418, 40){\makebox(0,0){$h^{\ast}$}}
\put(418, 60){\makebox(0,0){$h^{\ast}$}}
\put(385, 1){\makebox(0,0)[I]{Dual-h $(2d)$}}

\end{picture}
\end{figure}

Observe that in all these four cases for $\widehat{HFL}(-s_{1}, -s_{2})$, $HFL^{-}(-s_{1}+1, -s_{2}+1)=0$, so $d_{2}=0$ in the spectral sequence corresponding to  $\widehat{HFL}(-s_{1}, -s_{2})$. Now the computation of $\widehat{HFL}(-s_{1}, -s_{2})$ is quite straightforward. 

In Dual-h $(2a)$, \[\widehat{HFL}(-s_{1}, -s_{2}) :\simeq \left[
\begin{tikzcd}
\mathbb{F}[-2h^{\ast}+1] \arrow[swap]{d}{U_{2}}  & 0 \arrow{l}{U_{1}} \arrow[swap]{d}{U_{2}}  \\%
\mathbb{F}[-2h^{\ast}]  & \mathbb{F}[-2h^{\ast}+1] \arrow{l}{U_{1}}
 \end{tikzcd}
\right] \]
By the grading reason, $d_{2}=U_{1}=U_{2}=0$. Then it is easy to obtain $\widehat{HFL}(-s_{1}, -s_{2})\cong \F[-2h^{\ast}]\oplus \F[-2h^{\ast}]\oplus \F[-2h^{\ast}]$ and the Euler characteristic $\chi= 3$. By symmetry,  $\widehat{HFL}(s_{1}, s_{2})$ should contain 3 generators with same total gradings. Observe that $HFL^{-}(s_{1}, s_{2})=\F[-2h-2]\oplus \F[-2h-3]$. By grading reason, the generator with total grading $-2h-2$ survives in $\widehat{HFL}(s_{1}, s_{2})$. Thus $\widehat{HFL}(s_{1}, s_{2})\cong \F[-2h-2]\oplus\F[-2h-2]\oplus \F[-2h-2]$.  The Euler characteristic $\chi=3$

In Dual-h $(2b)$, \[\widehat{HFL}(-s_{1}, -s_{2}) :\simeq \left[
\begin{tikzcd}
0 \arrow[swap]{d}{U_{2}}  & 0 \arrow{l}{U_{1}} \arrow[swap]{d}{U_{2}}  \\%
\mathbb{F}[-2h^{\ast}]  & 0 \arrow{l}{U_{1}}
 \end{tikzcd}
\right] \]
In this case, $\widehat{HFL}(-s_{1}, -s_{2})\cong \F[-2h^{\ast}]$. By the similar argument in Dual-h $(2a)$,  $\widehat{HFL}(L)(s_{1}, s_{2})\cong \mathbb{F}[-2h-2]$. The Euler Characteristic $\chi=1$.

 In Dual-h $(2c)$, \[\widehat{HFL}(-s_{1}, -s_{2}) :\simeq \left[
\begin{tikzcd}
\mathbb{F}[-2h^{\ast}+1] \arrow[swap]{d}{U_{2}}  & 0 \arrow{l}{U_{1}} \arrow[swap]{d}{U_{2}}  \\%
\mathbb{F}[-2h^{\ast}] & 0  \arrow{l}{U_{1}}
 \end{tikzcd}
\right] \]
By degree reason, $d_{2}=U_{1}=U_{2}=0$.Then $\widehat{HFL}(-s_{1}, -s_{2})\cong \mathbb{F}[-2h^{\ast}]\oplus \mathbb{F}[-2h^{\ast}]$. So $\widehat{HFL}(s_{1}, s_{2})\cong \F[-2h-2]\oplus \F[-2h-2]$ and the Euler Characteristic is $\chi=2$. 

In Dual-h $(2d)$, 
\[\widehat{HFL}(-s_{1}, -s_{2}) :\simeq \left[
\begin{tikzcd}
0  \arrow[swap]{d}{U_{2}}  & 0 \arrow{l}{U_{1}} \arrow[swap]{d}{U_{2}}  \\%
\mathbb{F}[-2h^{\ast}]  & \mathbb{F}[-2h^{\ast}+1]  \arrow{l}{U_{1}}
 \end{tikzcd}
\right] \]
So $\widehat{HFL}(L)(s_{1}, s_{2})\cong \mathbb{F}[-2h-2]\oplus \mathbb{F}[-2h-2]$ by similar argument above and the Euler Characteristic is $\chi=2$. 

Now we can consider Case 3. In this case, $HFL^{-}(s_{1}+1, s_{2}+1) \cong \F[-2h]\oplus \F[-2h-1]$. Then there are three possibilities for $HFL^{-}(s_{1}, s_{2})$ if $d_{2}$ may be nontrivial. $HFL^{-}(s_{1}, s_{2})$ is either $\F[-2h-4]$ or $\F[-2h-4]\oplus \F[-2h-5]$ or $\F[-2h-3]$. If $HFL^{-}(s_{1}, s_{2})=\F[-2h-4]$, its  $h$-function is shown in  Case $(3a)$ and if $HFL^{-}(s_{1}, s_{2})\cong \F[-2h-4]\oplus \F[-2h-5]$, its $h$-function is shown in Case $(3b)$ :
\begin{figure}[H]
\begin{picture}(100,70)(120,0)
\put(10,10){\framebox(110,60)}
\put(100, 20){\makebox(0,0){$h+2$}}
\put(65, 20){\makebox(0,0){$h+3$}}
\put(30, 20){\makebox(0,0){$h+3$}}
\put(100, 40){\makebox(0,0){$h+1$}}
\put(65, 40){\makebox(0,0){$h+2$}}
\put(30, 40){\makebox(0,0){$h+3$}}
\put(100, 60){\makebox(0,0){$h$}}
\put(65, 60){\makebox(0,0){$h+1$}}
\put(30, 60){\makebox(0,0){$h+2$}}
\put(60, 1){\makebox(0,0)[I]{Case $(3a)$}}

\put(180,10){\framebox(110,60)}
\put(200, 20){\makebox(0,0){$h+4$}}
\put(270, 20){\makebox(0,0){$h+2$}}
\put(235, 20){\makebox(0,0){$h+3$}}
\put(200, 40){\makebox(0,0){$h+3$}}
\put(270, 40){\makebox(0,0){$h+1$}}
\put(235, 40){\makebox(0,0){$h+2$}}
\put(200, 60){\makebox(0,0){$h+2$}}
\put(270, 60){\makebox(0,0){$h$}}
\put(235, 60){\makebox(0,0){$h+1$}}
\put(230, 1){\makebox(0,0)[I]{Case $(3b)$}}
\end{picture}
\end{figure}

In Case $(3a)$ and Case $(3b)$, we observe that both  generators in $HFL^{-}(s_{1}+1, s_{2}+1)$ have nontrivial images in $HFL^{-}(s_{1}, s_{2}+1)$ under $U_{1}$ and in $HFL^{-}(s_{1}+1, s_{2})$ under $U_{2}$ by Lemma \ref{U-action}. So the two generators have nontrivial images under the differential  $d_{1}=U$ and cannot survive in $E^{2}$-page. Thus $d_{2}$ is trivial in both cases. 

 If $HFL^{-}(s_{1}, s_{2})\cong \F[-2h-3]$, there are four possibilities for the  $h$-function corresponding to $\widehat{HFL}(s_{1}, s_{2})$:

\begin{figure}[H]
\begin{picture}(100,70)(165,0)
\put(1,10){\framebox(100,60)}
\put(15, 20){\makebox(0,0){$h+3$}}
\put(15, 40){\makebox(0,0){$h+2$}}
\put(15, 60){\makebox(0,0){$h+1$}}
\put(50, 20){\makebox(0,0){$h+2$}}
\put(50, 40){\makebox(0,0){$h+2$}}
\put(50, 60){\makebox(0,0){$h+1$}}
\put(85, 20){\makebox(0,0){$h+1$}}
\put(85, 40){\makebox(0,0){$h+1$}}
\put(85, 60){\makebox(0,0){$h$}}
\put(55, 1){\makebox(0,0)[I]{Case $(3c)$}}

\put(110,10){\framebox(100,60)}
\put(125, 20){\makebox(0,0){$h+3$}}
\put(125, 40){\makebox(0,0){$h+2$}}
\put(125, 60){\makebox(0,0){$h+2$}}
\put(160, 20){\makebox(0,0){$h+2$}}
\put(160, 40){\makebox(0,0){$h+2$}}
\put(160, 60){\makebox(0,0){$h+1$}}
\put(195, 20){\makebox(0,0){$h+2$}}
\put(195, 40){\makebox(0,0){$h+1$}}
\put(195, 60){\makebox(0,0){$h$}}
\put(165, 1){\makebox(0,0)[I]{Case $(3d)$}}

\put(220, 10){\framebox(100, 60)}
\put(235, 20){\makebox(0,0){$h+3$}}
\put(235, 40){\makebox(0,0){$h+2$}}
\put(235, 60){\makebox(0,0){$h+1$}}
\put(270, 20){\makebox(0,0){$h+2$}}
\put(270, 40){\makebox(0,0){$h+2$}}
\put(270, 60){\makebox(0,0){$h+1$}}
\put(305, 20){\makebox(0,0){$h+2$}}
\put(305, 40){\makebox(0,0){$h+1$}}
\put(305, 60){\makebox(0,0){$h$}}
\put(275, 1){\makebox(0,0)[I]{Case $(3e)$}}

\put(330, 10){\framebox(100, 60)}
\put(345, 20){\makebox(0,0){$h+3$}}
\put(345, 40){\makebox(0,0){$h+2$}}
\put(345, 60){\makebox(0,0){$h+2$}}
\put(380, 20){\makebox(0,0){$h+2$}}
\put(380, 40){\makebox(0,0){$h+2$}}
\put(380, 60){\makebox(0,0){$h+1$}}
\put(415, 20){\makebox(0,0){$h+1$}}
\put(415, 40){\makebox(0,0){$h+1$}}
\put(415, 60){\makebox(0,0){$h$}}
\put(385, 1){\makebox(0,0)[I]{Case $(3f)$}}

\end{picture}
\end{figure}

Let $h^{\ast}=h(-s_{1}, -s_{2})=h(s_{1}, s_{2})+s_{1}+s_{2}$. By similar argument above, we find the $h$-function of $\widehat{HFL}(-s_{1}, -s_{2})$ corresponding to each case:

\begin{figure}[H]
\begin{picture}(100,70)(165,5)
\put(1,20){\framebox(100,60)}
\put(15, 30){\makebox(0,0){$h^{\ast}$}}
\put(15, 50){\makebox(0,0){$h^{\ast}$}}
\put(17, 70){\makebox(0,0){$h^{\ast}-1$}}
\put(50, 30){\makebox(0,0){$h^{\ast}$}}
\put(50, 50){\makebox(0,0){$h^{\ast}$}}
\put(50, 70){\makebox(0,0){$h^{\ast}-1$}}
\put(85, 30){\makebox(0,0){$h^{\ast}-1$}}
\put(85, 50){\makebox(0,0){$h^{\ast}-1$}}
\put(85, 70){\makebox(0,0){$h^{\ast}-1$}}
\put(55, 11){\makebox(0,0)[I]{Dual-h $(3c)$}}

\put(110,20){\framebox(100,60)}
\put(125, 30){\makebox(0,0){$h^{\ast}$}}
\put(125, 50){\makebox(0,0){$h^{\ast}$}}
\put(125, 70){\makebox(0,0){$h^{\ast}$}}
\put(160, 30){\makebox(0,0){$h^{\ast}$}}
\put(160, 50){\makebox(0,0){$h^{\ast}$}}
\put(160, 70){\makebox(0,0){$h^{\ast}-1$}}
\put(195, 30){\makebox(0,0){$h^{\ast}$}}
\put(195, 50){\makebox(0,0){$h^{\ast}-1$}}
\put(195, 70){\makebox(0,0){$h^{\ast}-1$}}
\put(165, 11){\makebox(0,0)[I]{Dual-h $(3d)$}}

\put(220, 20){\framebox(100, 60)}
\put(235, 30){\makebox(0,0){$h^{\ast}$}}
\put(235, 50){\makebox(0,0){$h^{\ast}$}}
\put(235, 70){\makebox(0,0){$h^{\ast}$}}
\put(270, 30){\makebox(0,0){$h^{\ast}$}}
\put(270, 50){\makebox(0,0){$h^{\ast}$}}
\put(270, 70){\makebox(0,0){$h^{\ast}-1$}}
\put(305, 30){\makebox(0,0){$h^{\ast}-1$}}
\put(305, 50){\makebox(0,0){$h^{\ast}-1$}}
\put(305, 70){\makebox(0,0){$h^{\ast}-1$}}
\put(275, 11){\makebox(0,0)[I]{Dual-h $(3e)$}}

\put(330, 20){\framebox(100, 60)}
\put(345, 30){\makebox(0,0){$h^{\ast}$}}
\put(345, 50){\makebox(0,0){$h^{\ast}$}}
\put(347, 70){\makebox(0,0){$h^{\ast}-1$}}
\put(380, 30){\makebox(0,0){$h^{\ast}$}}
\put(380, 50){\makebox(0,0){$h^{\ast}$}}
\put(380, 70){\makebox(0,0){$h^{\ast}-1$}}
\put(415, 30){\makebox(0,0){$h^{\ast}$}}
\put(415, 50){\makebox(0,0){$h^{\ast}-1$}}
\put(415, 70){\makebox(0,0){$h^{\ast}-1$}}
\put(385, 11){\makebox(0,0)[I]{Dual-h $(3f)$}}

\end{picture}
\end{figure}

Observe that in all four cases, $HFL^{-}(-s_{1}, -s_{2})=0$. So $d_{2}$ is trivial in the spectral sequence corresponding to $\widehat{HFL}(-s_{1}, -s_{2})$. We compute $\widehat{HFL}(-s_{1}, -s_{2})$ and therefore $\widehat{HFL}(s_{1}, s_{2})$. 

In Dual-h$(3c)$, $\widehat{HFL}(-s_{1}, -s_{2})\cong \F[-2h^{\ast}+1]$. By symmetry, $\widehat{HFL}(s_{1}, s_{2})\cong \F[-2h-3]$ with Euler Characteristic  $\chi=-1$. 

In Case$(3d)$, $\widehat{HFL}(L)(s_{1}, s_{2})\cong \mathbb{F}[-2h-3]\oplus \mathbb{F}[-2h-3] \oplus \mathbb{F}[-2h-3]$ and the Euler Characteristic is $\chi=-3$ by similar computation above. 

In Case$(3e)$, $\widehat{HFL}(L)(s_{1}, s_{2})\cong \mathbb{F}[-2h-3]\oplus \mathbb{F}[-2h-3] $ and the Euler Characteristic is $\chi=-2$.

 In Case$(3f)$, $\widehat{HFL}(L)(s_{1}, s_{2})\cong \mathbb{F}[-2h-3]\oplus \mathbb{F}[-2h-3] $ and the Euler Characteristic is $\chi=-2$. 

Thus we  conclude that for any $L$-space link $L=L_{1}\cup L_{2}$ with two components, once the  $h$-function is determined, we can compute $\widehat{HFL}(s_{1}, s_{2})$ with any $(s_{1}, s_{2})\in \bH$. By the equations in Section 2.2, the $h$-function is determined by Alexander polynomials $\Delta_{L}(x_{1}, x_{2})$, $\Delta_{L_{1}}(t)$, $\Delta_{L_{2}}(t)$ and the linking number $lk(L_{1}, L_{2})$. 
\qed

Furthermore, We  also get a bound for rank$_{\F}(\widehat{HFL}(s_{1}, s_{2}))$ and the Euler characteristic $\chi(\widehat{HFL}(s_{1}, s_{2}))$ with any $(s_{1}, s_{2})\in \bH$.\\

\noindent
{\bf Proof of Corollary \ref{rank}:}
Consider the following short exact sequence:
\begin{equation}
\label{C(1)}
0\rightarrow CFL^{-}(s_{1}+1, s_{2}+1) \stackrel{U_{1}}{\longrightarrow} CFL^{-}(s_{1}, s_{2}+1)\rightarrow C_{1}(s_{1}, s_{2}+1)\rightarrow 0
\end{equation}
where $C_{1}(s_{1}, s_{2}+1)$ is the quotient complex with $(s_{1}, s_{2}+1)\in \bH$. By the computation of $\widehat{HFL}(s_{1}, s_{2})$, we have
\begin{equation}
\label{C(3)}
\widehat{CFL}(s_{1}, s_{2})\cong C_{1}(s_{1}, s_{2})/ U_{2}(C_{1}(s_{1}, s_{2}+1))
\end{equation}
Now we claim that rank$_{\F}(H_{\ast}(C_{1}(s_{1}, s_{2}+1)))\leq 2$ with any $(s_{1}, s_{2})\in \bH$. From the short exact sequence (\ref{C(1)}), we know that 
\begin{equation}
\label{C(2)}
\textup{rank}_{\F}(H_{\ast}(C_{1}(s_{1}, s_{2}+1)))\leq \textup{rank}_{\F}(HFL^{-}(s_{1}+1, s_{2}+1))+\textup{rank}_{\F}(HFL^{-}(s_{1}, s_{2}+1))
\end{equation}
If $\textup{rank}_{\F}(H_{\ast}(C_{1}(s_{1}, s_{2}+1)))\geq 3$, then at least one of $HFL^{-}(s_{1}+1, s_{2}+1)$ and $HFL^{-}(s_{1}, s_{2}+1)$ should have rank at least $2$, and the other one should have rank at least $1$. By the computation in Section 2.1, the possible $h$-functions corresponding to $HFL^{-}(s_{1}+1, s_{2}+1)$ and $HFL^{-}(s_{1}, s_{2}+1)$ are as follows:

\begin{figure}[H]
\begin{picture}(100,40)(165,40)
\put(10,40){\framebox(100,40)}

\put(25, 50){\makebox(0,0){$h+1$}}
\put(25, 70){\makebox(0,0){$h+1$}}

\put(60, 50){\makebox(0,0){$h+1$}}
\put(60, 70){\makebox(0,0){$h$}}

\put(95, 50){\makebox(0,0){$h$}}
\put(95, 70){\makebox(0,0){$h-1$}}
\put(65, 31){\makebox(0,0)[I]{Case $1$}}

\put(160,40){\framebox(100,40)}
\put(175, 50){\makebox(0,0){$h+2$}}
\put(175, 70){\makebox(0,0){$h+1$}}

\put(210, 50){\makebox(0,0){$h+1$}}
\put(210, 70){\makebox(0,0){$h$}}

\put(245, 50){\makebox(0,0){$h$}}
\put(245, 70){\makebox(0,0){$h$}}
\put(215, 31){\makebox(0,0)[I]{Case $2$}}

\put(310, 40){\framebox(100, 40)}

\put(325, 50){\makebox(0,0){$h+2$}}
\put(325, 70){\makebox(0,0){$h+1$}}

\put(360, 50){\makebox(0,0){$h+1$}}
\put(360, 70){\makebox(0,0){$h$}}

\put(395, 50){\makebox(0,0){$h$}}
\put(395, 70){\makebox(0,0){$h-1$}}
\put(365, 31){\makebox(0,0)[I]{Case $3$}}

\end{picture}
\end{figure}

Here we assume that the unique generator of $H_{\ast}(A^{-}(s_{1}, s_{2}+1))$ has homological grading $-2h$. In Case $1$, we have $U_{1}:\F[-2h+2]\oplus \F[-2h+1] \rightarrow \F[-2h]$. Let $\alpha$ denote the generator of $\F[-2h+2]\subseteq HFL^{-}(s_{1}+1, s_{2}+1)$ and let $\beta$ denote the generator of $\F[-2h]\cong HFL^{-}(s_{1}, s_{2}+1)$. By Lemma \ref{U-action}, $U(\alpha)=\beta$. Then $H_{\ast}(C_{1}(s_{1}))\cong \F[-2h+1]$. Its rank in Case $1$ is $1$. In Case $2$, we have $U_{1}: \F[-2h+1]\rightarrow \F[-2h]\oplus \F[-2h-1]$. By the similar argument, $H_{\ast}(C_{1}(s_{1}, s_{2}+1))\cong \F[-2h]$. Then it has rank $1$ in this case. In Case $3$, we have $U_{1}: \F[-2h+2]\oplus \F[-2h+1]\rightarrow \F[-2h]\oplus \F[-2h-1]$. The image of the generator of $\F[-2h+2]$ is the generator of $\F[-2h]$ and the image of the generator of $\F[-2h+1]$ is the generator of $\F[-2h-1]$. So $H_{\ast}(C_{1}(s_{1}, s_{2}+1))=0$. Thus with any $(s_{1}, s_{2})\in \bH$, rank$_{\F}(H_{\ast}(C_{1}(s_{1}, s_{2}+1)))\leq 2$. By Equation (\ref{C(3)}), rank$_{\F}(\widehat{HFL}(s_{1}, s_{2}))\leq \textup{rank}_{\F}(H_{\ast}(C_{1}(s_{1}, s_{2}+1)))+\textup{rank}_{\F}(H_{\ast}(C_{1}(s_{1}, s_{2})))\leq 2+2=4$ with any $(s_{1}, s_{2})\in \bH$. Therefore  $-4 \leq\chi(\widehat{HFL}(L, s_{1}, s_{2})) \leq 4$.
\qed

In fact, for the first example given in the proof of Theorem \ref{thm 2} where $d_{2}=0$, we have  $\chi(\widehat{HFL}(L, s_{1}, s_{2}))= -4$. Similarly, we can construct an example so that $\chi(\widehat{HFL}(L, s_{1}, s_{2}))= 4$. 
 
\begin{example}
Assume that the $h$-function corresponding to $\widehat{HFL}(s_{1}, s_{2})$ is as follows:
\begin{figure}[H]
\begin{picture}(100,60)(165,10)
\put(150,10){\framebox(100,60)}
\put(165, 20){\makebox(0,0){$h+2$}}
\put(165, 40){\makebox(0,0){$h+2$}}
\put(165, 60){\makebox(0,0){$h+1$}}
\put(200, 20){\makebox(0,0){$h+2$}}
\put(200, 40){\makebox(0,0){$h+1$}}
\put(200, 60){\makebox(0,0){$h+1$}}
\put(235, 20){\makebox(0,0){$h+1$}}
\put(235, 40){\makebox(0,0){$h+1$}}
\put(235, 60){\makebox(0,0){$h$}}
\end{picture}
\end{figure}

In this case, $\widehat{HFL}(s_{1}, s_{2})\cong \F[-2h-2]\oplus \F[-2h-2]\oplus \F[-2h-2] \oplus \F[-2h-2]$. Thus $\chi(\widehat{HFL}(s_{1}, s_{2}))=4$. 

\end{example}

\begin{example}
The Heegaard Floer link homology $\widehat{HFL}$ of the two-bridge link $b(20, -3)$. 

\begin{figure}[H]
\centering

\definecolor{linkcolor0}{rgb}{0.45, 0.15, 0.15}
\definecolor{linkcolor1}{rgb}{0.15, 0.15, 0.45}
\begin{tikzpicture}[line width=1, line cap=round, line join=round]
  \begin{scope}[color=linkcolor0]
    \draw (8.7, 2.36) .. controls (8.9, 2.59) and (8.93, 3.22) .. 
          (8.95, 3.68) .. controls (9, 4.77) and (6.42, 4.8) .. 
          (4.38, 4.83) .. controls (2.35, 4.85) and (0.18, 4.88) .. 
          (0.18, 4.00) .. controls (0.19, 3.49) and (0.19, 2.94) .. (0.56, 2.59);
    \draw (0.56, 2.59) .. controls (0.83, 2.33) and (1.20, 2.15) .. (1.42, 2.39);
    \draw (1.57, 2.55) .. controls (1.77, 2.78) and (2.13, 2.64) .. (2.38, 2.40);
    \draw (2.38, 2.40) .. controls (2.60, 2.21) and (2.91, 2.14) .. (3.10, 2.32);
    \draw (3.25, 2.40) .. controls (3.45, 2.71) and (3.81, 2.51) .. (4.06, 2.33); 
    \draw (4.06, 2.33) .. controls (4.28, 2.14) and (4.59, 2.07) .. (4.68, 2.25); 
   \draw (4.83, 2.41) .. controls (5.14, 2.71) and (5.71, 2.71) .. 
          (6.2, 2.70) .. controls (6.75, 2.70) and (7.33, 2.63) .. (7.75, 2.27);
    \draw (7.75, 2.27) .. controls (8, 2.07) and (8.34, 2.02) .. (8.57, 2.22);
   
    \draw[->] (0.3, 3.00) -- +(0.01, -0.05);
  \end{scope}
  \begin{scope}[color=linkcolor1]
     \draw (6.1, 1.16) .. controls (5.23, 0.49) and (4.19, 0.08) .. 
          (3.1, 0.13) .. controls (0.91, 0.17) and (0.08, 0.28) .. 
          (0.12, 0.92) .. controls (0.14, 1.21) and (0.16, 1.62) .. 
          (0.18, 1.80) .. controls (0.20, 1.93) and (0.21, 2.06) .. 
          (0.21, 2.19) .. controls (0.21, 2.36) and (0.34, 2.48) .. (0.49, 2.55);
    \draw (0.66, 2.64) .. controls (0.93, 2.77) and (1.26, 2.67) .. (1.50, 2.47);
    \draw (1.50, 2.47) .. controls (1.75, 2.26) and (2.09, 2.13) .. (2.31, 2.33);
    \draw (2.46, 2.47) .. controls (2.65, 2.66) and (2.96, 2.58) .. (3.18, 2.40);
    \draw (3.18, 2.40) .. controls (3.43, 2.19) and (3.77, 2.06) .. (3.99, 2.26);
    \draw (4.14, 2.40) .. controls (4.33, 2.59) and (4.64, 2.51) .. (4.86, 2.33); 
    \draw (4.86, 2.33) .. controls (5.31, 1.94) and (5.77, 1.56) .. (6, 1.24);
     
     \draw (6.15, 1.08) .. controls (6.62, 0.38) and (6.3, 0.22) .. 
          (7.39, 0.20) .. controls (8.36, 0.17) and (9.46, 0.22) .. 
          (9.44, 1.04) .. controls (9.42, 1.62) and (9.12, 1.97) .. (8.65, 2.31);

    \draw (8.65, 2.31) .. controls (8.28, 2.61) and (7.93, 2.61) .. (7.68, 2.42);
    \draw (7.53, 2.27) .. controls (7.05, 1.90) and (6.57, 1.53) .. (6.1, 1.16);
    \draw[->] (0.22, 2.00) -- +(0.01, 0.05);
  \end{scope}
\end{tikzpicture}
\caption{$b(20,  -3)$}
\end{figure}

Yajing Liu proved that the two-bridge link $L=b(20, -3)$ is an $L$-space link \cite[Theorem 3.8]{Liu}. Its two link components are both unknots with linking number $2$. Its normalized multi-variable Alexander polynomial is  \cite{ND}:
\begin{equation}
\begin{split}
\Delta_{L}(t_{1}, t_{2})= & t_{1}^{1/2}t_{2}^{3/2}+t_{1}^{3/2}t_{2}^{1/2}+t_{1}^{1/2}t_{2}^{-1/2}+t_{1}^{-1/2}t_{2}^{1/2}+t_{1}^{-3/2}t_{2}^{-1/2}+t_{1}^{-1/2}t_{2}^{-3/2} \\
& -t_{1}^{3/2}t_{2}^{3/2}-t_{1}^{1/2}t_{2}^{1/2}-t_{1}^{-1/2}t_{2}^{-1/2}-t_{1}^{-3/2}t_{2}^{-3/2}
\end{split}
\end{equation}

\begin{figure}[H]
\centering
\begin{tikzpicture}
\draw[help lines, color=gray!30, dashed] (-4.9,-4.9) grid (4.9,4.9);
\draw[->, color=gray, thick] (-5,0)--(5,0) node[right]{$s_{1}$};
\draw[->, color=gray, thick] (0,-5)--(0,5) node[above]{$s_{2}$};
\node[color=red] at (0.1, 0) {2};
\node[color=red] at (0.1, 1) {1};
\node[color=red] at (0.1, 2) {1};
\node[color=red] at (0.1, 3) {1};
\node[color=red] at (0.1, 4) {1};
\node[color=red] at (0.1, 5) {$\vdots$};
\node[color=red] at (0.1, -1) {2};
\node[color=red] at (0.2, -2) {3};
\node[color=red] at (0.1, -3) {4};
\node[color=red] at (0.1, -4) {5};
\node[color=red] at (0.1, -5) {$\vdots$};

\node[color=red] at (1.1, 0) {1};
\node[color=red] at (1.1, 1) {1};
\node[color=red] at (1.1, 2) {0};
\node[color=red] at (1.1, 3) {0};
\node[color=red] at (1.1, 4) {0};
\node[color=red] at (1.1, 5) {$\vdots$};
\node[color=red] at (1.1, -1) {2};
\node[color=red] at (1.1, -2) {3};
\node[color=red] at (1.1, -3) {4};
\node[color=red] at (1.1, -4) {5};
\node[color=red] at (1.1, -5) {$\vdots$};

\node[color=red] at (2.1, 0) {1};
\node[color=red] at (2.1, 1) {0};
\node[color=red] at (2.1, 2) {0};
\node[color=red] at (2.1, 3) {0};
\node[color=red] at (2.1, 4) {0};
\node[color=red] at (2.1, 5) {$\vdots$};
\node[color=red] at (2.1, -1) {2};
\node[color=red] at (2.1, -2) {3};
\node[color=red] at (2.1, -3) {4};
\node[color=red] at (2.1, -4) {5};
\node[color=red] at (2.1, -5) {$\vdots$};

\node[color=red] at (3.1, 0) {1};
\node[color=red] at (3.1, 1) {0};
\node[color=red] at (3.1, 2) {0};
\node[color=red] at (3.1, 3) {0};
\node[color=red] at (3.1, 4) {0};
\node[color=red] at (3.1, 5) {$\vdots$};
\node[color=red] at (3.1, -1) {2};
\node[color=red] at (3.1, -2) {3};
\node[color=red] at (3.1, -3) {4};
\node[color=red] at (3.1, -4) {5};
\node[color=red] at (3.1, -5) {$\vdots$};

\node[color=red] at (4.1, 0) {1};
\node[color=red] at (4.1, 1) {0};
\node[color=red] at (4.1, 2) {0};
\node[color=red] at (4.1, 3) {0};
\node[color=red] at (4.1, 4) {0};
\node[color=red] at (4.1, 5) {$\vdots$};
\node[color=red] at (4.1, -1) {2};
\node[color=red] at (4.1, -2) {3};
\node[color=red] at (4.1, -3) {4};
\node[color=red] at (4.1, -4) {5};
\node[color=red] at (4.1, -5) {$\vdots$};

\node[color=red] at (5, 0) {$\cdots$};
\node[color=red] at (5, 1) {$\cdots$};
\node[color=red] at (5, 2) {$\cdots$};
\node[color=red] at (5, 3) {$\cdots$};
\node[color=red] at (5, 4) {$\cdots$};
\node[color=red] at (5, 5) {$\vdots$};
\node[color=red] at (5, -1) {$\cdots$};
\node[color=red] at (5, -2) {$\cdots$};
\node[color=red] at (5, -3) {$\cdots$};
\node[color=red] at (5, -4) {$\cdots$};
\node[color=red] at (5, -5) {$\vdots$};

\node[color=red] at (-0.9, 0) {2};
\node[color=red] at (-0.9, 1) {2};
\node[color=red] at (-0.9, 2) {2};
\node[color=red] at (-0.9, 3) {2};
\node[color=red] at (-0.9, 4) {2};
\node[color=red] at (-0.9, 5) {$\vdots$};
\node[color=red] at (-0.8, -1) {3};
\node[color=red] at (-0.8, -2) {3};
\node[color=red] at (-0.9, -3) {4};
\node[color=red] at (-0.9, -4) {5};
\node[color=red] at (-0.9, -5) {$\vdots$};

\node[color=red] at (-1.8, 0) {3};
\node[color=red] at (-1.9, 1) {3};
\node[color=red] at (-1.9, 2) {3};
\node[color=red] at (-1.9, 3) {3};
\node[color=red] at (-1.9, 4) {3};
\node[color=red] at (-1.9, 5) {$\vdots$};
\node[color=red] at (-1.8, -1) {3};
\node[color=red] at (-1.8, -2) {4};
\node[color=red] at (-1.9, -3) {5};
\node[color=red] at (-1.9, -4) {6};
\node[color=red] at (-1.9, -5) {$\vdots$};

\node[color=red] at (-2.9, 0) {4};
\node[color=red] at (-2.9, 1) {4};
\node[color=red] at (-2.9, 2) {4};
\node[color=red] at (-2.9, 3) {4};
\node[color=red] at (-2.9, 4) {4};
\node[color=red] at (-2.9, 5) {$\vdots$};
\node[color=red] at (-2.9, -1) {4};
\node[color=red] at (-2.9, -2) {5};
\node[color=red] at (-2.9, -3) {6};
\node[color=red] at (-2.9, -4) {7};
\node[color=red] at (-2.9, -5) {$\vdots$};

\node[color=red] at (-4, 0) {5};
\node[color=red] at (-4, 1) {5};
\node[color=red] at (-4, 2) {5};
\node[color=red] at (-4, 3) {5};
\node[color=red] at (-4, 4) {5};
\node[color=red] at (-4, 5) {$\vdots$};
\node[color=red] at (-4, -1) {5};
\node[color=red] at (-4, -2) {6};
\node[color=red] at (-4, -3) {7};
\node[color=red] at (-4, -4) {8};
\node[color=red] at (-4, -5) {$\vdots$};

\node[color=red] at (-5, 0) {$\cdots$};
\node[color=red] at (-5, 1) {$\cdots$};
\node[color=red] at (-5, 2) {$\cdots$};
\node[color=red] at (-5, 3) {$\cdots$};
\node[color=red] at (-5, 4) {$\cdots$};
\node[color=red] at (-5, 5) {$\vdots$};
\node[color=red] at (-5, -1) {$\cdots$};
\node[color=red] at (-5, -2) {$\cdots$};
\node[color=red] at (-5, -3) {$\cdots$};
\node[color=red] at (-5, -4) {$\cdots$};
\node[color=red] at (-5, -5) {$\vdots$};

\node at (0, 2) {$\bullet$};
\node at (0, 1) {$\bullet$};
\node at (0, 0) {$\bullet$};
\node at (0, -1) {$\bullet$};
\node at (0, -2) {$\bullet$};

\node at (1, 2) {$\bullet$};
\node at (1, 1) {$\bullet$};
\node at (1, 0) {$\bullet$};
\node at (1, -1) {$\bullet$};

\node at (2, 2) {$\bullet$};
\node at (2, 1) {$\bullet$};
\node at (2, 0) {$\bullet$};

\node at (-1, -2) {$\bullet$};
\node at (-1, 1) {$\bullet$};
\node at (-1, 0) {$\bullet$};
\node at (-1, -1) {$\bullet$};

\node at (-2, -2) {$\bullet$};
\node at (-2, -1) {$\bullet$};
\node at (-2, 0) {$\bullet$};

\end{tikzpicture} 
\caption{$h$-function for $b(20, -3)$}
\end{figure}

Let $L_{1}$ and $L_{2}$ denote the unknot components. We obtain the normalized Alexander polynomials of $L_{1}$ and $L_{2}$:
$$\dfrac{t}{t-1} \Delta_{L_{1}}(t)=\dfrac{t}{t-1} \Delta_{L_{2}}(t)=1+t^{-1}+t^{-2}+t^{-3}+t^{-4}+\cdots$$

Using results of Section 2.2, we compute the $h$-function for $\widehat{HFL}(s_{1}, s_{2})$ with any $(s_{1}, s_{2})\in \bH$ by the above Alexander polynomials. The $h$-function is shown in Figure 4. The  numbers denote $h(s_{1}, s_{2})$ for any $(s_{1}, s_{2})\in \bH$. For example $h(0, 0)=h(-1, 0)=2$ and the black dot $\bullet$ denotes the  lattice point $(s_{1}, s_{2})\in \bH$ where $\widehat{HFL}(s_{1}, s_{2})$ is nonzero. By explicit computation, the link Floer homology $\widehat{HFL}(s_{1}, s_{2})$ is shown in Figure 5 for any $(s_{1}, s_{2})\in \bH$. 
 We observe that  $| \chi(s_{1}, s_{2})|=$rank$_{\F}(\widehat{HFL}(s_{1}, s_{2}))$ for any $(s_{1}, s_{2})\in \bH$ and the rank of $\widehat{HFL}(s_{1}, s_{2})$ ranges from $0$ to $4$ for this link $L$. This indicates that the bound for the rank in  Corollary \ref{rank} can be realized by some $L$-space link with some $(s_{1}, s_{2})\in \bH$. Here rank$_{\F}(\widehat{HFL}(2, 2))=1$,   rank$_{\F}(\widehat{HFL}(2, 1))=2$,  rank$_{\F}(\widehat{HFL}(1, 0))=3$ and rank$_{\F}(\widehat{HFL}(0, 0))=4$, rank$_{\F}(\widehat{HFL}(3, 0))=0$.

\begin{figure}[H]
\centering
\begin{tikzpicture}
\draw[help lines, color=gray!30, dashed] (-2.9,-2.9) grid (2.9,2.9);
\draw[->, color=gray, thick] (-3,0)--(3,0) node[right]{$s_{1}$};
\draw[->, color=gray, thick] (0,-3)--(0,3) node[above]{$s_{2}$};

\node at (0, 2) {$\F$};
\node at (0, 1) {$\F^{3}$};
\node at (0, 0) {$\F^{4}$};
\node at (0, -1) {$\F^{3}$};
\node at (0, -2) {$\F$};

\node at (1, 2) {$\F^{2}$};
\node at (1, 1) {$\F^{4}$};
\node at (1, 0) {$\F^{3}$};
\node at (1, -1) {$\F$};

\node at (2, 2) {$\F$};
\node at (2, 1) {$\F^{2}$};
\node at (2, 0) {$\F$};

\node at (-1, -2) {$\F^{2}$};
\node at (-1, 1) {$\F$};
\node at (-1, 0) {$\F^{3}$};
\node at (-1, -1) {$\F^{4}$};

\node at (-2, -2) {$\F$};
\node at (-2, -1) {$\F^{2}$};
\node at (-2, 0) {$\F$};

\end{tikzpicture} 
\caption{$\widehat{HFL}(b(20, -3))$}
\end{figure}
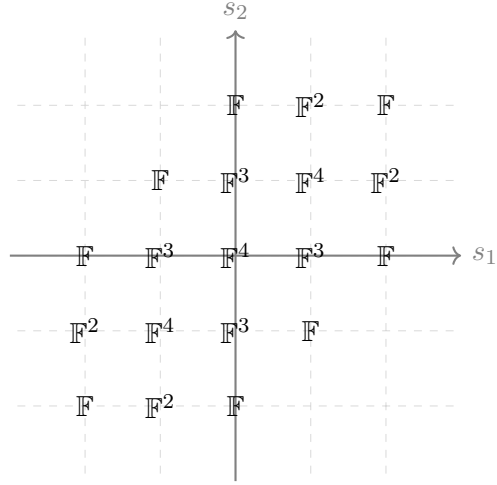

\end{example}

\section{Application to Thurston norm}

P.Ozsv\'ath and Z.Szab\'o showed that Heegaard Floer link homology  detects the Thurston norm of the link complement \cite{oss}. In Section 3, for any $L$-space link $L=L_{1}\cup L_{2}$ with two components and any $\mathbf{s}\in \bH$, we  computed $\widehat{HFL}(L, \mathbf{s})$ by using the Alexander polynomials $\Delta_{L}(t_{1}, t_{2})$, $\Delta_{L_{1}}(t)$, $\Delta_{L_{2}}(t)$ and the linking number $lk(L_{1}, L_{2})$. So we can compute the link Floer homology polytope for  the  link $L$ and also compute its dual Thurston polytope  and the Thurston (semi-)norm  \cite[Theorem 1.1]{oss}. 

In Section 1, we introduced complexity $\chi_{-}(F)$ for any compact oriented surface $F$ with boundary. For any link $L\subseteq S^{3}$, and any homology class $h\in H_{2}(S^{3}, L)$,  we can assign a function:
$$x(h)=\min \limits_{F\hookrightarrow S^{3}-\textup{nb}(L), [F]=h} \chi_{-}(F)$$
This function can be naturally extended to  a semi-norm, the \emph{Thurston semi-norm}, denoted by $x: H_{2}(S^{3}, L; \R)\rightarrow \R$.

\begin{theorem}\cite[Theorem 1]{Th}
The function $x: H_{2}(S^{3}, L; \R)\rightarrow \R$ is a semi-norm that vanishes exactly on the subspace spanned by embedded surfaces of non-negative Euler characteristic. 
\end{theorem}

Assume that $L\subseteq S^{3}$ is a link with $l$ components on $S^{3}$. Let $u_{i}$ denote the meridian of the $i^{th}$-component $L_{i}$ of $L$. Recall that  every lattice point $\bf{s}\in \bH$ can be written as 
$$\sum\limits_{i=1}^{l} s_{i}\cdot [u_{i}]$$
where $s_{i}\in \Q$ satisfies the property that 
$$2s_{i}+lk(L_{i}, L-L_{i})$$ 
is an even integer for $i=1, \cdots, l$.

In \cite{oss} the Heegaard Floer link homology  provides a function $y: H^{1}(S^{3}-L; \R)\rightarrow \R$ defined by the formula 
$$y(h)=\max\limits_{ \lbrace\mathbf{s}\in \bH\subseteq H_{1}(S^{3}-L; \R)|\widehat{HFL}(L, \mathbf{s})\neq 0\rbrace} |\langle s, h\rangle |$$
A trivial component of a link $L$ is an unknot component which is also unlinked from the rest of the link.

\begin{theorem} \cite[Theorem 1.1]{oss}
\label{thm4}
For an oriented link $L\subseteq S^{3}$ with no trivial components, its Heegaard Floer link homology detects the Thurston (semi-)norm  of the link complement. For each $h\in H^{1}(S^{3}-L; \R)$, we have 
$$x(\textup{PD}[h])+\sum\limits_{i=1}^{l}|\langle h, u_{i} \rangle |=2y(h)$$
where $u_{i}$ is the meridian of the $i$-th component of $L$, and $|\langle h, u_{i}\rangle|$ denotes the absolute value of the Kronecker paring of $h\in H^{1}(S^{3}-L; \R)$ and $u_{i}\in H_{1}(S^{3}-L; \R)$.  

\end{theorem}

The unit ball for the norm $x$ is called \emph{Thurston polytope}, and the unit ball for the norm $y$ is called \emph{link Floer homology polytope}, which is also the convex hull of those  $\mathbf{s}\in \bH$ for which $\widehat{HFL}(L, \mathbf{s})\neq 0$. The unit ball for the dual norm $x^{\ast}$ of $x$ in $H_{1}(S^{3}-L; \R)$ is called \emph{dual Thurston polytope}. By Theorem \ref{thm4}, twice link Floer homology polytope can be written as the sum of the dual Thurston polytope and an element of the symmetric hypercube in $H^{1}(S^{3}-L)$ with edge-length two \cite{oss}.  Next, we will give some examples of  $L$-space links with two components and compute their link Floer homology polytopes by using  the Alexander polynomials and linking numbers  in detail. Moreover, we will compute the dual Thurston polytopes and Thurston norms of link complements by Theorem \ref{thm4}. We also compare  the link Floer homology polytope and the convex hull of those $\mathbf{s}\in \bH$ for which $\chi(\widehat{HFL}(L, \mathbf{s}))\neq 0$.  

\begin{example}
Dual Thurston polytope for $L$-space link $L=L7n1$.

The link $L7n1$ in Figure 6 is an $L$-space link \cite[Example 3.17]{Liu}. The link component $L_{1}$ is an unknot and the other link component $L_{2}$ is a right-handed trefoil. The linking number is $2$ and the multi-variable Alexander polynomial is:
$$\Delta_{L}(t_{1}, t_{2})=t_{1}^{1/2}t_{2}^{3/2}+t_{1}^{-1/2}t_{2}^{-3/2}$$
The normalized Alexander polynomials of $L_{1}$ and $L_{2}$ are:
$$\dfrac{t}{t-1} \Delta_{L_{1}}(t)=1+t^{-1}+t^{-2}+t^{-3}+t^{-4}+\cdots$$
$$\dfrac{t}{t-1}\Delta_{L_{2}}(t)=t+t^{-1}+t^{-2}+t^{-3}+t^{-4}+\cdots$$

\begin{figure}[H]
\centering
\includegraphics[width=2.0in]{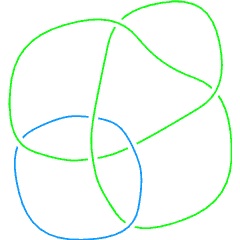} 
\caption{L7n1}
\end{figure}

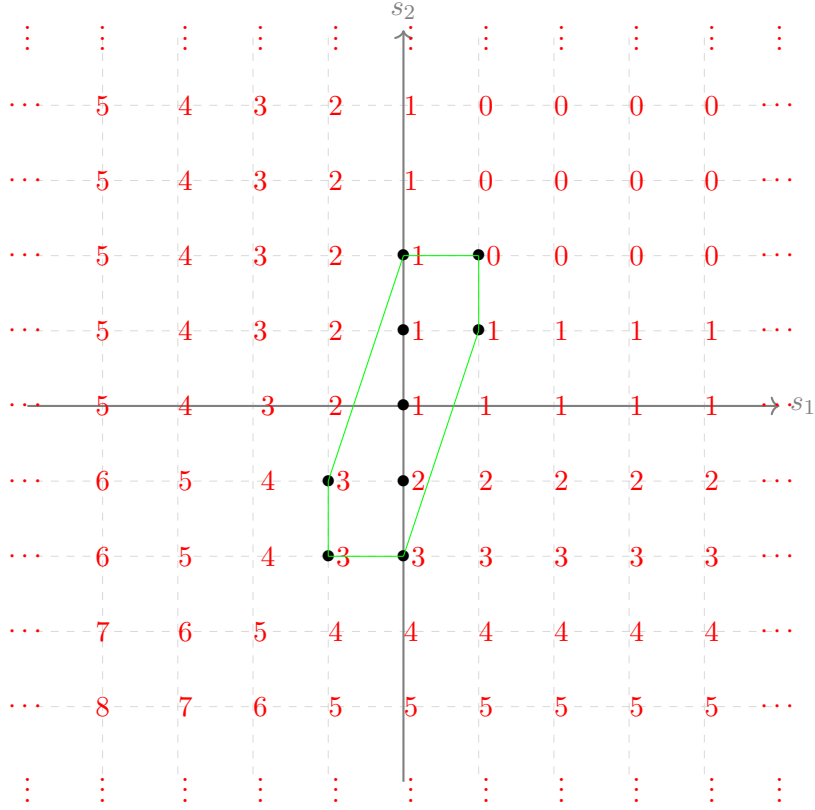
\begin{figure}[H]
\centering 
\begin{tikzpicture}
\draw[help lines, color=gray!30, dashed] (-4.9,-4.9) grid (4.9,4.9);
\draw[->, color=gray, thick] (-5,0)--(5,0) node[right]{$s_{1}$};
\draw[->, color=gray, thick] (0,-5)--(0,5) node[above]{$s_{2}$};
\node[color=red] at (0.2, 0) {1};
\node[color=red] at (0.2, 1) {1};
\node[color=red] at (0.2, 2) {1};
\node[color=red] at (0.1, 3) {1};
\node[color=red] at (0.1, 4) {1};
\node[color=red] at (0.1, 5) {$\vdots$};
\node[color=red] at (0.2, -1) {2};
\node[color=red] at (0.2, -2) {3};
\node[color=red] at (0.1, -3) {4};
\node[color=red] at (0.1, -4) {5};
\node[color=red] at (0.1, -5) {$\vdots$};

\node[color=red] at (1.1, 0) {1};
\node[color=red] at (1.2, 1) {1};
\node[color=red] at (1.2, 2) {0};
\node[color=red] at (1.1, 3) {0};
\node[color=red] at (1.1, 4) {0};
\node[color=red] at (1.1, 5) {$\vdots$};
\node[color=red] at (1.1, -1) {2};
\node[color=red] at (1.1, -2) {3};
\node[color=red] at (1.1, -3) {4};
\node[color=red] at (1.1, -4) {5};
\node[color=red] at (1.1, -5) {$\vdots$};

\node[color=red] at (2.1, 0) {1};
\node[color=red] at (2.1, 1) {1};
\node[color=red] at (2.1, 2) {0};
\node[color=red] at (2.1, 3) {0};
\node[color=red] at (2.1, 4) {0};
\node[color=red] at (2.1, 5) {$\vdots$};
\node[color=red] at (2.1, -1) {2};
\node[color=red] at (2.1, -2) {3};
\node[color=red] at (2.1, -3) {4};
\node[color=red] at (2.1, -4) {5};
\node[color=red] at (2.1, -5) {$\vdots$};

\node[color=red] at (3.1, 0) {1};
\node[color=red] at (3.1, 1) {1};
\node[color=red] at (3.1, 2) {0};
\node[color=red] at (3.1, 3) {0};
\node[color=red] at (3.1, 4) {0};
\node[color=red] at (3.1, 5) {$\vdots$};
\node[color=red] at (3.1, -1) {2};
\node[color=red] at (3.1, -2) {3};
\node[color=red] at (3.1, -3) {4};
\node[color=red] at (3.1, -4) {5};
\node[color=red] at (3.1, -5) {$\vdots$};

\node[color=red] at (4.1, 0) {1};
\node[color=red] at (4.1, 1) {1};
\node[color=red] at (4.1, 2) {0};
\node[color=red] at (4.1, 3) {0};
\node[color=red] at (4.1, 4) {0};
\node[color=red] at (4.1, 5) {$\vdots$};
\node[color=red] at (4.1, -1) {2};
\node[color=red] at (4.1, -2) {3};
\node[color=red] at (4.1, -3) {4};
\node[color=red] at (4.1, -4) {5};
\node[color=red] at (4.1, -5) {$\vdots$};

\node[color=red] at (5, 0) {$\cdots$};
\node[color=red] at (5, 1) {$\cdots$};
\node[color=red] at (5, 2) {$\cdots$};
\node[color=red] at (5, 3) {$\cdots$};
\node[color=red] at (5, 4) {$\cdots$};
\node[color=red] at (5, 5) {$\vdots$};
\node[color=red] at (5, -1) {$\cdots$};
\node[color=red] at (5, -2) {$\cdots$};
\node[color=red] at (5, -3) {$\cdots$};
\node[color=red] at (5, -4) {$\cdots$};
\node[color=red] at (5, -5) {$\vdots$};

\node[color=red] at (-0.9, 0) {2};
\node[color=red] at (-0.9, 1) {2};
\node[color=red] at (-0.9, 2) {2};
\node[color=red] at (-0.9, 3) {2};
\node[color=red] at (-0.9, 4) {2};
\node[color=red] at (-0.9, 5) {$\vdots$};
\node[color=red] at (-0.8, -1) {3};
\node[color=red] at (-0.8, -2) {3};
\node[color=red] at (-0.9, -3) {4};
\node[color=red] at (-0.9, -4) {5};
\node[color=red] at (-0.9, -5) {$\vdots$};

\node[color=red] at (-1.8, 0) {3};
\node[color=red] at (-1.9, 1) {3};
\node[color=red] at (-1.9, 2) {3};
\node[color=red] at (-1.9, 3) {3};
\node[color=red] at (-1.9, 4) {3};
\node[color=red] at (-1.9, 5) {$\vdots$};
\node[color=red] at (-1.8, -1) {4};
\node[color=red] at (-1.8, -2) {4};
\node[color=red] at (-1.9, -3) {5};
\node[color=red] at (-1.9, -4) {6};
\node[color=red] at (-1.9, -5) {$\vdots$};

\node[color=red] at (-2.9, 0) {4};
\node[color=red] at (-2.9, 1) {4};
\node[color=red] at (-2.9, 2) {4};
\node[color=red] at (-2.9, 3) {4};
\node[color=red] at (-2.9, 4) {4};
\node[color=red] at (-2.9, 5) {$\vdots$};
\node[color=red] at (-2.9, -1) {5};
\node[color=red] at (-2.9, -2) {5};
\node[color=red] at (-2.9, -3) {6};
\node[color=red] at (-2.9, -4) {7};
\node[color=red] at (-2.9, -5) {$\vdots$};

\node[color=red] at (-4, 0) {5};
\node[color=red] at (-4, 1) {5};
\node[color=red] at (-4, 2) {5};
\node[color=red] at (-4, 3) {5};
\node[color=red] at (-4, 4) {5};
\node[color=red] at (-4, 5) {$\vdots$};
\node[color=red] at (-4, -1) {6};
\node[color=red] at (-4, -2) {6};
\node[color=red] at (-4, -3) {7};
\node[color=red] at (-4, -4) {8};
\node[color=red] at (-4, -5) {$\vdots$};

\node[color=red] at (-5, 0) {$\cdots$};
\node[color=red] at (-5, 1) {$\cdots$};
\node[color=red] at (-5, 2) {$\cdots$};
\node[color=red] at (-5, 3) {$\cdots$};
\node[color=red] at (-5, 4) {$\cdots$};
\node[color=red] at (-5, 5) {$\vdots$};
\node[color=red] at (-5, -1) {$\cdots$};
\node[color=red] at (-5, -2) {$\cdots$};
\node[color=red] at (-5, -3) {$\cdots$};
\node[color=red] at (-5, -4) {$\cdots$};
\node[color=red] at (-5, -5) {$\vdots$};

\node at (0, 2) {$\bullet$};
\node at (0, 1) {$\bullet$};
\node at (0, 0) {$\bullet$};
\node at (0, -1) {$\bullet$};
\node at (0, -2) {$\bullet$};

\node at (1, 2) {$\bullet$};
\node at (1, 1) {$\bullet$};

\node at (-1, -2) {$\bullet$};
\node at (-1, -1) {$\bullet$};

\draw[color=green] (-1, -1)--(0,2);
\draw[color=green] (0, 2)--(1, 2);
\draw[color=green] (1, 2)--(1, 1);
\draw[color=green] (1, 1)--(0, -2);
\draw[color=green] (0, -2)--(-1, -2);
\draw[color=green] (-1, -2)--(-1, -1);

\end{tikzpicture} 
\caption{h-function for $L7n1$}
\end{figure}

\begin{figure}[H]
\centering 
\begin{tikzpicture}
\draw[help lines, color=gray!30, dashed] (-2.9,-2.9) grid (2.9,2.9);
\draw[->, color=gray, thick] (-2,0)--(2,0) node[right]{$s_{1}$};
\draw[->, color=gray, thick] (0,-3)--(0,3) node[above]{$s_{2}$};

\node at (0, 2) {$\F$};
\node at (0, 1) {$\F$};
\node at (0, 0) {$\F^{2}$};
\node at (0, -1) {$\F$};
\node at (0, -2) {$\F$};

\node at (1, 2) {$\F$};
\node at (1, 1) {$\F$};

\node at (-1, -2) {$\F$};
\node at (-1, -1) {$\F$};

\draw[color=green] (-1, -1)--(0,2);
\draw[color=green] (0, 2)--(1, 2);
\draw[color=green] (1, 2)--(1, 1);
\draw[color=green] (1, 1)--(0, -2);
\draw[color=green] (0, -2)--(-1, -2);
\draw[color=green] (-1, -2)--(-1, -1);

\end{tikzpicture} 
\caption{Link Floer homology polytope for $L7n1$}
\end{figure}
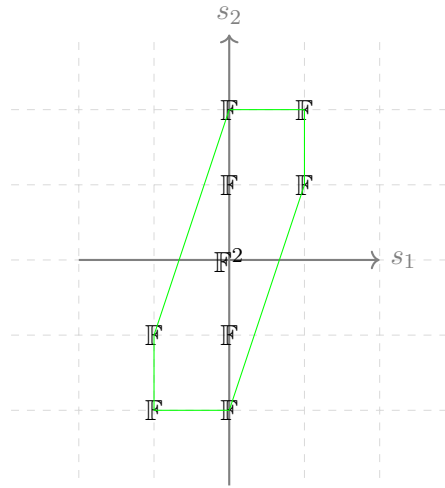

The $h$-function corresponding to $\widehat{HFL}(s_{1}, s_{2})$ with any $(s_{1}, s_{2})\in \bH$ is shown in Figure 7. In this figure, the  numbers denote the $h$-function, and $\bullet$ denotes the lattice point $(s_{1}, s_{2})\in \bH$ where $\widehat{HFL}(s_{1}, s_{2})$ is nonzero. By explicit computation, the link Floer homology $\widehat{HFL}(s_{1}, s_{2})$ is shown in Figure 8. 
 Moreover, $\widehat{HFL}(0, 0)\cong \F[-2]\oplus \F[-3]$, so $\chi(\widehat{HFL}(0, 0))$ is zero. For any other lattice point $(s_{1}, s_{2})$ denoted by $\bullet$ except $(0,0)$, $\widehat{HFL}(s_{1}, s_{2})$ has rank one and $\chi(\widehat{HFL}(s_{1}, s_{2}))$ is also nonzero. Thus in this example, the link Floer homology polytope  is the same with the convex hull of those $(s_{1}, s_{2})\in \bH$ for which $\chi(\widehat{HFL}(s_{1}, s_{2}))$ are nonzero. By Theorem \ref{thm4}, the dual Thurston polytope on $H_{1}(S^{3}-L; \R)$ is in Figure 9. 

\begin{figure}[H]
\begin{tikzpicture}
\draw (-1,1.5)--(1, 1.5);
\draw (-1, 0.5)--(1, 0.5);
\draw (-1, -0.5)--(1,-0.5);
\draw (-1,-1.5)--(1, -1.5);

\draw (-0.5,2)--(-0.5,-2);
\draw (0.5, 2)--(0.5,-2);

\filldraw [thick,  color=red]
(-0.5, -1.5) circle (3pt) -- (0.5, 1.5) circle (3pt);

\end{tikzpicture}
\caption{ Dual Thurston  polytope for $L7n1$}
\label{L3}
\end{figure}
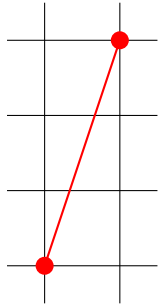

Here the thick red line is the dual Thurston polytope for $L7n1$. Observe that the dual Thurston polytope is the same with the Newton polytope of the Alexander polynomial $\Delta_{L}(t_{1}, t_{2})$. The unknot component of $L7n1$ bounds a surface $F_{L_{1}}$ with Euler characteristic $-1$ and the right-handed trefoil link component $L_{2}$ bounds a surface $F_{L_{2}}$ with Euler characteristic $-3$. The surfaces $F_{L_{1}}$ and $F_{L_{2}}$ have maximal Euler characteristic in their respective homology classes. 
\end{example}

\begin{example}
Dual Thurston polytope for pretzel link $L=b(-2, 3, 8)$

\begin{figure}[H]
\centering
\includegraphics[width=2.0in]{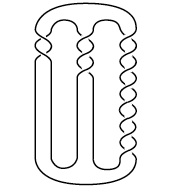} 
\caption{b(-2, 3, 8)}
\end{figure}

We claim that the pretzel link $b(-2, 3, 8)$  is an $L$-space link. It has two components. The link component $L_{1}$ is an unknot and the other link component $L_{2}$ is a right-handed trefoil shown in Figure 10. The linking number of this pretzel link is $5$. Let $P_{1}$ be the knot obtained from $b(-2, 3, 8)$ by $1$-Dehn surgery on unknotted component $L_{1}$. The knot $P_{1}$ is the $T(5, 6; 2, 1)$ twisted torus knot \cite[Proposition 3.1]{RR}. This twisted torus knot is an $L$-space knot proved by F. Vafaee \cite[Theorem 1]{Vafa}. Suppose that $d$ is sufficiently large. Then $S_{1, d}^{3}(L)=S_{d-25}^{3}(P_{1})$ is an $L$-space. The link components $L_{1}$ and $L_{2}$ are $L$-space knots, so $S_{1}^{3}(L_{1})$ and $S_{d}^{3}(L_{2})$ are both $L$-spaces. We also have $d-25>0$, so the pretzel link $b(-2, 3, 8)$ is an $L$-space link by $L$-space surgery criterion \cite[Lemma 2.6]{Liu}. The symmetrized Alexander polynomial of $b(-2, 3, 8)$ is:
$$\Delta_{L}(t_{1}, t_{2})=t_{1}^{-2}t_{2}^{-3}+t_{1}^{-1}t_{2}^{-2}+1+t_{1}t_{2}+t_{1}^{2}t_{2}^{3}$$

\begin{figure}[H]
\centering 
\begin{tikzpicture}
\draw[help lines, color=gray!30, dashed, xshift=-0.5cm, yshift=0.5cm] (-2.9,-4.4) grid (3.9,3.9);
\draw[->, color=gray, thick] (-3,0)--(3,0) node[right]{$s_{1}$};
\draw[->, color=gray, thick] (0,-4)--(0,4) node[above]{$s_{2}$};

\node at (2.5, 3.5) {$\F$};
\node at (2.5, 2.5) {$\F$};
\node at (1.5, 3.5) {$\F$};
\node at (1.5, 2.5) {$\F$};
\node at (1.5, 0.5) {$\F$};
\node at (1.5, 1.5) {$\F$};
\node at (0.5, 1.5) {$\F$};
\node at (0.5, 0.5) {$\F^{2}$};
\node at (0.5, -0.5) {$\F$};
\node at (-0.5, 0.5) {$\F$};
\node at (-0.5, -0.5) {$\F^{2}$};
\node at (-0.5, -1.5) {$\F$};
\node at (-1.5, -0.5) {$\F$};
\node at (-1.5, -1.5) {$\F$};
\node at (-1.5, -2.5) {$\F$};
\node at (-1.5, -3.5) {$\F$};
\node at (-2.5, -3.5) {$\F$};
\node at (-2.5, -2.5) {$\F$};

\draw[color=green] (-2.5, -3.5)--(-2.5, -2.5);
\draw[color=green] (-2.5, -2.5)--(-1.5, -0.5);
\draw[color=green] (-1.5, -0.5)--(1.5,  3.5);
\draw[color=green] (1.5, 3.5)--( 2.5, 3.5);
\draw[color=green] (2.5, 3.5)--( 2.5, 2.5);
\draw[color=green] (2.5, 2.5)--( 1.5, 0.5);
\draw[color=green] (1.5, 0.5)--(-1.5,-3.5);
\draw[color=green] (-1.5, -3.5)--(-2.5,-3.5);

\end{tikzpicture} 
\caption{Link Floer homology polytope for $b(-2, 3, 8)$}
\end{figure}
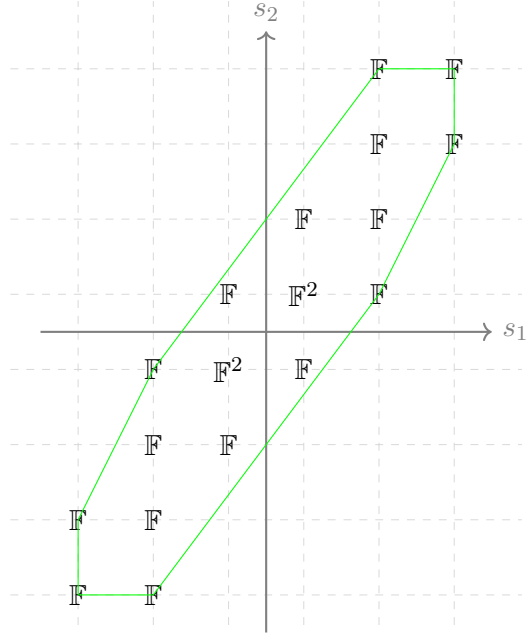

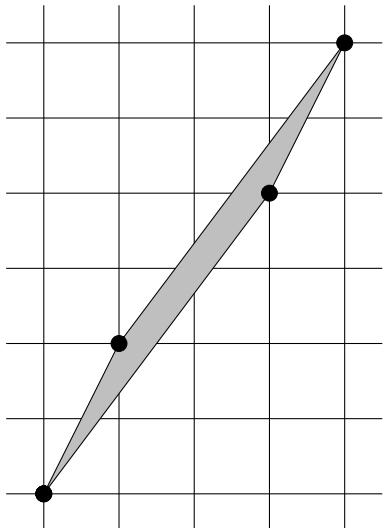
\begin{figure}[H]
\begin{tikzpicture}
\draw (-2.5,-3)--(2.5, -3);
\draw (-2.5,-2)--(2.5, -2);
\draw (-2.5,-1)--(2.5, -1);
\draw (-2.5, 0)--(2.5, 0);
\draw (-2.5,1)--(2.5, 1);
\draw (-2.5,2)--(2.5, 2);
\draw (-2.5,3)--(2.5, 3);

\draw (-2,-3.5)--(-2, 3.5);
\draw (-1,-3.5)--(-1, 3.5);
\draw (0,-3.5)--(0, 3.5);
\draw (1,-3.5)--(1, 3.5);
\draw (2,-3.5)--(2, 3.5);

\fill [fill=lightgray] 
(-2,-3)--(-1,-1)--(2,3)--(1,1)--(-2,-3);

\filldraw
(-2, -3) circle (3pt) -- (-1, -1) circle (3pt) -- (2, 3) circle (3pt) -- (1, 1) circle (3pt) -- (-2, -3) circle (3pt);
 \end{tikzpicture}
\caption{Dual Thurston polytope for $b(-2, 3, 8)$}
\label{L2}
\end{figure}

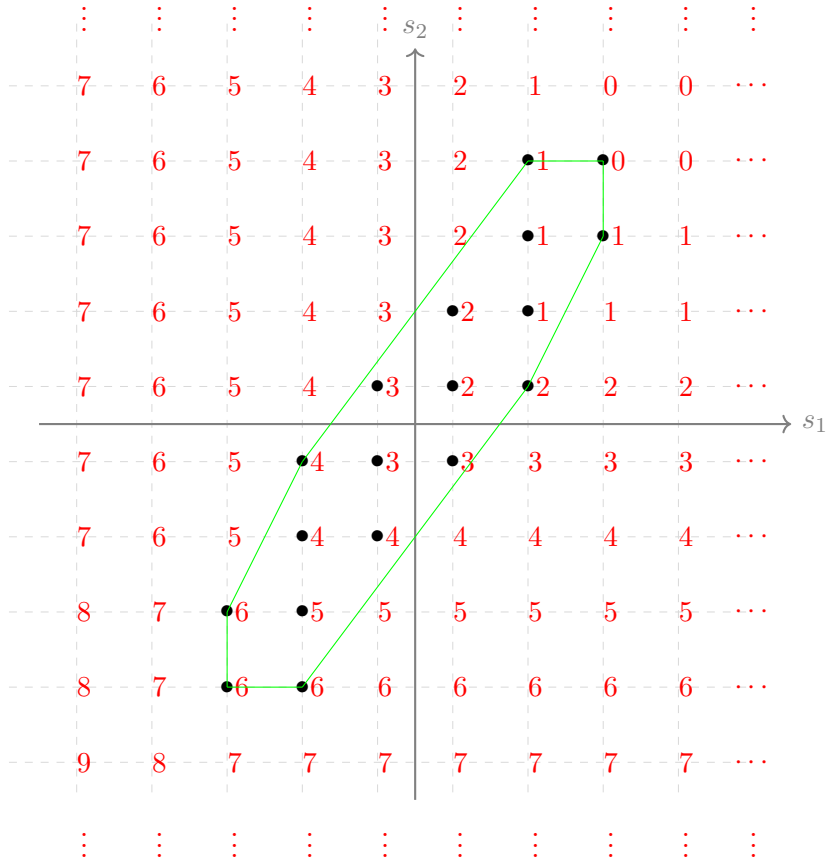
\begin{figure}[H]
\centering 
\begin{tikzpicture}
\draw[help lines, color=gray!30, dashed, xshift=-0.5cm, yshift=0.5cm] (-4.9,-5.4) grid (4.9,4.9);
\draw[->, color=gray, thick] (-5,0)--(5,0) node[right]{$s_{1}$};
\draw[->, color=gray, thick] (0,-5)--(0,5) node[above]{$s_{2}$};
\node[color=red] at (0.7, 0.5) {2};
\node[color=red] at (0.7, 1.5) {2};
\node[color=red] at (0.6, 2.5) {2};
\node[color=red] at (0.6, 3.5) {2};
\node[color=red] at (0.6, 4.5) {2};
\node[color=red] at (0.6, 5.5) {$\vdots$};
\node[color=red] at (0.7, -0.5) {3};
\node[color=red] at (0.6, -1.5) {4};
\node[color=red] at (0.6, -2.5) {5};
\node[color=red] at (0.6, -3.5) {6};
\node[color=red] at (0.6, -4.5) {7};
\node[color=red] at (0.6, -5.5) {$\vdots$};

\node[color=red] at (1.7, 0.5) {2};
\node[color=red] at (1.7, 1.5) {1};
\node[color=red] at (1.7, 2.5) {1};
\node[color=red] at (1.7, 3.5) {1};
\node[color=red] at (1.6, 4.5) {1};
\node[color=red] at (1.6, 5.5) {$\vdots$};
\node[color=red] at (1.6, -0.5) {3};
\node[color=red] at (1.6, -1.5) {4};
\node[color=red] at (1.6, -2.5) {5};
\node[color=red] at (1.6, -3.5) {6};
\node[color=red] at (1.6, -4.5) {7};
\node[color=red] at (1.6, -5.5) {$\vdots$};

\node[color=red] at (2.6, 0.5) {2};
\node[color=red] at (2.6, 1.5) {1};
\node[color=red] at (2.7, 2.5) {1};
\node[color=red] at (2.7, 3.5) {0};
\node[color=red] at (2.6, 4.5) {0};
\node[color=red] at (2.6, 5.5) {$\vdots$};
\node[color=red] at (2.6, -0.5) {3};
\node[color=red] at (2.6, -1.5) {4};
\node[color=red] at (2.6, -2.5) {5};
\node[color=red] at (2.6, -3.5) {6};
\node[color=red] at (2.6, -4.5) {7};
\node[color=red] at (2.6, -5.5) {$\vdots$};

\node[color=red] at (3.6, 0.5) {2};
\node[color=red] at (3.6, 1.5) {1};
\node[color=red] at (3.6, 2.5) {1};
\node[color=red] at (3.6, 3.5) {0};
\node[color=red] at (3.6, 4.5) {0};
\node[color=red] at (3.6, 5.5) {$\vdots$};
\node[color=red] at (3.6, -0.5) {3};
\node[color=red] at (3.6, -1.5) {4};
\node[color=red] at (3.6, -2.5) {5};
\node[color=red] at (3.6, -3.5) {6};
\node[color=red] at (3.6, -4.5) {7};
\node[color=red] at (3.6, -5.5) {$\vdots$};

\node[color=red] at (4.5, 0.5) {$\cdots$};
\node[color=red] at (4.5, 1.5) {$\cdots$};
\node[color=red] at (4.5, 2.5) {$\cdots$};
\node[color=red] at (4.5, 3.5) {$\cdots$};
\node[color=red] at (4.5, 4.5) {$\cdots$};
\node[color=red] at (4.5, 5.5) {$\vdots$};
\node[color=red] at (4.5, -0.5) {$\cdots$};
\node[color=red] at (4.5, -1.5) {$\cdots$};
\node[color=red] at (4.5, -2.5) {$\cdots$};
\node[color=red] at (4.5, -3.5) {$\cdots$};
\node[color=red] at (4.5, -4.5) {$\cdots$};
\node[color=red] at (4.5, -5.5) {$\vdots$};

\node[color=red] at (-0.3, 0.5) {3};
\node[color=red] at (-0.4, 1.5) {3};
\node[color=red] at (-0.4, 2.5) {3};
\node[color=red] at (-0.4, 3.5) {3};
\node[color=red] at (-0.4, 4.5) {3};
\node[color=red] at (-0.4, 5.5) {$\vdots$};
\node[color=red] at (-0.3, -0.5) {3};
\node[color=red] at (-0.3, -1.5) {4};
\node[color=red] at (-0.4, -2.5) {5};
\node[color=red] at (-0.4, -3.5) {6};
\node[color=red] at (-0.4, -4.5) {7};
\node[color=red] at (-0.4, -5.5) {$\vdots$};

\node[color=red] at (-1.4, 0.5) {4};
\node[color=red] at (-1.4, 1.5) {4};
\node[color=red] at (-1.4, 2.5) {4};
\node[color=red] at (-1.4, 3.5) {4};
\node[color=red] at (-1.4, 4.5) {4};
\node[color=red] at (-1.4, 5.5) {$\vdots$};
\node[color=red] at (-1.3, -0.5) {4};
\node[color=red] at (-1.3, -1.5) {4};
\node[color=red] at (-1.3, -2.5) {5};
\node[color=red] at (-1.3, -3.5) {6};
\node[color=red] at (-1.4, -4.5) {7};
\node[color=red] at (-1.4, -5.5) {$\vdots$};

\node[color=red] at (-2.4, 0.5) {5};
\node[color=red] at (-2.4, 1.5) {5};
\node[color=red] at (-2.4, 2.5) {5};
\node[color=red] at (-2.4, 3.5) {5};
\node[color=red] at (-2.4, 4.5) {5};
\node[color=red] at (-2.4, 5.5) {$\vdots$};
\node[color=red] at (-2.4, -0.5) {5};
\node[color=red] at (-2.4, -1.5) {5};
\node[color=red] at (-2.3, -2.5) {6};
\node[color=red] at (-2.3, -3.5) {6};
\node[color=red] at (-2.4, -4.5) {7};
\node[color=red] at (-2.4, -5.5) {$\vdots$};

\node[color=red] at (-3.4, 0.5) {6};
\node[color=red] at (-3.4, 1.5) {6};
\node[color=red] at (-3.4, 2.5) {6};
\node[color=red] at (-3.4, 3.5) {6};
\node[color=red] at (-3.4, 4.5) {6};
\node[color=red] at (-3.4, 5.5) {$\vdots$};
\node[color=red] at (-3.4, -0.5) {6};
\node[color=red] at (-3.4, -1.5) {6};
\node[color=red] at (-3.4, -2.5) {7};
\node[color=red] at (-3.4, -3.5) {7};
\node[color=red] at (-3.4, -4.5) {8};
\node[color=red] at (-3.4, -5.5) {$\vdots$};

\node[color=red] at (-4.4, 0.5) {7};
\node[color=red] at (-4.4, 1.5) {7};
\node[color=red] at (-4.4, 2.5) {7};
\node[color=red] at (-4.4, 3.5) {7};
\node[color=red] at (-4.4, 4.5) {7};
\node[color=red] at (-4.4, 5.5) {$\vdots$};
\node[color=red] at (-4.4, -0.5) {7};
\node[color=red] at (-4.4, -1.5) {7};
\node[color=red] at (-4.4, -2.5) {8};
\node[color=red] at (-4.4, -3.5) {8};
\node[color=red] at (-4.4, -4.5) {9};
\node[color=red] at (-4.4, -5.5) {$\vdots$};

\node at (2.5, 3.5) {$\bullet$};
\node at (2.5, 2.5) {$\bullet$};
\node at (1.5, 3.5) {$\bullet$};
\node at (1.5, 2.5) {$\bullet$};
\node at (1.5, 0.5) {$\bullet$};
\node at (1.5, 1.5) {$\bullet$};
\node at (0.5, 1.5) {$\bullet$};
\node at (0.5, 0.5) {$\bullet$};
\node at (0.5, -0.5) {$\bullet$};
\node at (-0.5, 0.5) {$\bullet$};
\node at (-0.5, -0.5) {$\bullet$};
\node at (-0.5, -1.5) {$\bullet$};
\node at (-1.5, -0.5) {$\bullet$};
\node at (-1.5, -1.5) {$\bullet$};
\node at (-1.5, -2.5) {$\bullet$};
\node at (-1.5, -3.5) {$\bullet$};
\node at (-2.5, -3.5) {$\bullet$};
\node at (-2.5, -2.5) {$\bullet$};

\draw[color=green] (-2.5, -3.5)--(-2.5, -2.5);
\draw[color=green] (-2.5, -2.5)--(-1.5, -0.5);
\draw[color=green] (-1.5, -0.5)--(1.5,  3.5);
\draw[color=green] (1.5, 3.5)--( 2.5, 3.5);
\draw[color=green] (2.5, 3.5)--( 2.5, 2.5);
\draw[color=green] (2.5, 2.5)--( 1.5, 0.5);
\draw[color=green] (1.5, 0.5)--(-1.5,-3.5);
\draw[color=green] (-1.5, -3.5)--(-2.5,-3.5);

\end{tikzpicture} 
\caption{h-function for $b(-2, 3, 8)$}
\end{figure}

The $h$-function corresponding to $\widehat{HFL}(s_{1}, s_{2})$ with any $(s_{1}, s_{2})\in \bH$ is shown in Figure 13. By explicit computation, the link Floer homology $\widehat{HFL}(s_{1}, s_{2})$ is shown in Figure 11. In Figure 11, $\textup{rank}_{\F}(\widehat{HFL}(1/2, 1/2))=\chi(\widehat{HFL}(1/2, 1/2))=2$ and $\textup{rank}_{\F}(\widehat{HFL}(-1/2, -1/2))=\chi(\widehat{HFL}(-1/2, -1/2))=2$. The link Floer homology polytope is the same with the convex hull of those $(s_{1}, s_{2})\in \bH$ for which $\chi(\widehat{HFL}(s_{1}, s_{2}))$ are nonzero. By Theorem \ref{thm4}, the dual Thurston polytope for $b(-2, 3, 8)$ is the shaded area in Figure 12.

\end{example}

\begin{remark}
For the $L$-space links $L7n1$ and $b(-2, 3, 8)$, the Thurston polytopes are both dual to the Newton polytopes of their symmetrized Alexander polynomials $\Delta_{L}(t_{1}, t_{2})$. P. Ozsv\'ath and Z.Szab\'o point out that the Thurston polytope of an alternating link is dual to the Newton polytope of its multi-variable Alexander polynomial \cite{oss}. This is also true for $L$-space knot. A natural question is whether the Thurston polytope of an $L$-space link with two components (which is not a split union of two $L$-space knots) is dual to the Newton polytope of its symmetrized Alexander polynomial. 
\end{remark}

\section{Two-component $L$-space links with vanishing Alexander polynomials}

In Section 4, we have given examples of $L$-space links  where  the Thurston polytopes are dual to the Newton polytopes of their symmetrized Alexander polynomials. In this section, we mainly discuss $2$-component $L$-space links with vanishing Alexander polynomials, especially split $L$-space links with two components. Recall that  the multi-variable Alexander polynomials for split links are $0$. So the Newton polytopes for split $L$-space links are empty, but the link Floer homology polytopes may be nontrivial. To see this in detail, we need some lemmas first.

\begin{lemma}\cite[Example 1.13(A)]{Liu}
Split disjoint unions of $L$-space knots are $L$-space links.
\end{lemma}

\begin{lemma}\cite[Proposition 3.11]{BG}
\label{lem 1}
For a split  $L$-space link $L=L_{1}\sqcup L_{2}$ with two components which are both $L$-space knots and any  $(s_{1}, s_{2})\in \bH$, the $h$-function $h(s_{1}, s_{2})$ satisfies that 
$$h(s_{1}, s_{2})=h_{1}(s_{1})+h_{2}(s_{2})$$
where $h_{1}(s_{1})$ and $h_{2}(s_{2})$ denote the $h$-functions of knots $L_{1}$ and $L_{2}$ respectively.  
\end{lemma}

\begin{remark}
$L$-space knots can be regarded as special $L$-space links with just one component.  For any $L$-space knot $K\subseteq S^{3}$, we can associate a  chain complex $A^{-}(s_{1})$ filtered by Alexander grading  and $H_{\ast}(A^{-}(s_{1}))$ has a unique generator for any $s_{1}$. Let  $-2h(s_{1})$ be  the homological grading of this unique generator. 
\end{remark}

\begin{proposition}
Let $L=L_{1}\sqcup L_{2}$ be the split union of two $L$-space knots $L_{1}$ and $L_{2}$. Then $\widehat{HFL}(L, s_{1}, s_{2})\cong \widehat{HFL}(L_{1}, s_{1})\otimes \widehat{HFL}(L_{2}, s_{2})\otimes (\F\oplus\F_{(-1)})$ with any $(s_{1}, s_{2})\in \bH$. 
\end{proposition}

\begin{proof}
The proof is quite straightforward by using our computation of $\widehat{HFL}(s_{1}, s_{2})$ in Section 3.  For any $(s_{1}, s_{2})\in\bH$, the possible values for the $h$-function corresponding to $HFK^{-}(L_{1}, s_{1})$ are like:

\begin{figure}[H]
\begin{picture}(100,30)(150,20)
\put(5,40){\line(1,0){80}}
\put(20,39){\circle*{4}}
\put(45,39){\circle*{4}}
\put(70,39){\circle*{4}}
\put(18, 45){$a$}
\put(5, 30){$s_{1}-1$}
\put(43, 30){$s_{1}$}
\put(43, 45){$a$}
\put(61, 30){$s_{1}+1$}
\put(68, 45){$a$}
\put(20, 13) { Case (1)}

\put(115,40){\line(1,0){80}}
\put(130,39){\circle*{4}}
\put(155,39){\circle*{4}}
\put(180,39){\circle*{4}}
\put(117, 45){$a+1$}
\put(115, 30){$s_{1}-1$}
\put(153, 30){$s_{1}$}
\put(153, 45){$a$}
\put(171, 30){$s_{1}+1$}
\put(178, 45){$a$}
\put(130, 13) { Case (2)}

\put(220,40){\line(1,0){80}}
\put(235,39){\circle*{4}}
\put(260,39){\circle*{4}}
\put(285,39){\circle*{4}}
\put(233, 45){$a$}
\put(220, 30){$s_{1}-1$}
\put(258, 30){$s_{1}$}
\put(258, 45){$a$}
\put(276, 30){$s_{1}+1$}
\put(278, 45){$a-1$}
\put(235, 13) { Case (3)}

\put(330,40){\line(1,0){80}}
\put(345,39){\circle*{4}}
\put(370,39){\circle*{4}}
\put(405,39){\circle*{4}}
\put(333, 45){$a+1$}
\put(330, 30){$s_{1}-1$}
\put(368, 30){$s_{1}$}
\put(367, 45){$a$}
\put(386, 30){$s_{1}+1$}
\put(388, 45){$a-1$}
\put(345, 13) { Case (4)}

\end{picture}
\end{figure}

Here $h_{1}(s_{1})=a$ and $a$ is any positive integer. Observe that
$$H_{\ast}(A^{-}(s_{1})/A^{-}(s_{1}-1))\cong HFK^{-}(L_{1}, s_{1})$$
$$\cdots \rightarrow HFK^{-}_{i+2}(s_{1}+1)\xrightarrow{U} HFK^{-}_{i}(s_{1})\rightarrow \widehat{HFK}_{i}(s_{1})\rightarrow HFK^{-}_{i+1}(s_{1}+1)\xrightarrow{U} HFK^{-}_{i-1}(s_{1})\rightarrow \cdots$$
The long exact sequence is induced by the short exact sequence:
$$0\rightarrow CFK^{-}(s_{1}+1)\xrightarrow{U} CFK^{-}(s_{1}) \rightarrow \widehat{CFK}(s_{1}) \rightarrow 0$$
By the long exact sequence above, 

Case (1) $\widehat{HFK}(L_{1}, s_{1})\cong 0$ 

Case (2) $\widehat{HFK}(L_{1}, s_{1})\cong \F[-2a]$ 

Case (3) $\widehat{HFK}(L_{1}, s_{1})\cong \F[-2a+1]$

Case (4) $\widehat{HFK}(L_{1}, s_{1})\cong 0$

Similarly, for the link component $L_{2}$, we assume that $h_{2}(s_{2})=b$. There are also four possibilities for the $h$-function corresponding to $\widehat{HFK}(L_{2}, s_{2})$. By Lemma 5.2, $h(s_{1}, s_{2})=h_{1}(s_{1})+h_{2}(s_{2})$. We find that there are only four possibilities for the $h$-function such that $\widehat{HFL}(L, s_{1}, s_{2})\neq 0$ 

\begin{figure}[H]
\begin{picture}(100,70)(165,0)
\put(1,10){\framebox(100,60)}
\put(15, 20){\makebox(0,0){$h+2$}}
\put(15, 40){\makebox(0,0){$h+1$}}
\put(15, 60){\makebox(0,0){$h+1$}}
\put(50, 20){\makebox(0,0){$h+1$}}
\put(50, 40){\makebox(0,0){$h$}}
\put(50, 60){\makebox(0,0){$h$}}
\put(85, 20){\makebox(0,0){$h+1$}}
\put(85, 40){\makebox(0,0){$h$}}
\put(85, 60){\makebox(0,0){$h$}}
\put(55, 1){\makebox(0,0)[I]{Case $a$}}

\put(110,10){\framebox(100,60)}
\put(125, 20){\makebox(0,0){$h+2$}}
\put(125, 40){\makebox(0,0){$h+2$}}
\put(125, 60){\makebox(0,0){$h+1$}}
\put(160, 20){\makebox(0,0){$h+1$}}
\put(160, 40){\makebox(0,0){$h+1$}}
\put(160, 60){\makebox(0,0){$h$}}
\put(195, 20){\makebox(0,0){$h+1$}}
\put(195, 40){\makebox(0,0){$h+1$}}
\put(195, 60){\makebox(0,0){$h$}}
\put(165, 1){\makebox(0,0)[I]{Case $b$}}

\put(220, 10){\framebox(100, 60)}
\put(235, 20){\makebox(0,0){$h+2$}}
\put(235, 40){\makebox(0,0){$h+1$}}
\put(235, 60){\makebox(0,0){$h+1$}}
\put(270, 20){\makebox(0,0){$h+2$}}
\put(270, 40){\makebox(0,0){$h+1$}}
\put(270, 60){\makebox(0,0){$h+1$}}
\put(305, 20){\makebox(0,0){$h+1$}}
\put(305, 40){\makebox(0,0){$h$}}
\put(305, 60){\makebox(0,0){$h$}}
\put(275, 1){\makebox(0,0)[I]{Case $c$}}

\put(330, 10){\framebox(100, 60)}
\put(345, 20){\makebox(0,0){$h+2$}}
\put(345, 40){\makebox(0,0){$h+2$}}
\put(345, 60){\makebox(0,0){$h+1$}}
\put(380, 20){\makebox(0,0){$h+2$}}
\put(380, 40){\makebox(0,0){$h+2$}}
\put(380, 60){\makebox(0,0){$h+1$}}
\put(415, 20){\makebox(0,0){$h+1$}}
\put(415, 40){\makebox(0,0){$h+1$}}
\put(415, 60){\makebox(0,0){$h$}}
\put(385, 1){\makebox(0,0)[I]{Case $d$}}

\end{picture}
\end{figure}

In Case $a$, the $h$-functions for link components $L_{1}$ and $L_{2}$ are both like Case (2): $(a+1) \quad a \quad a$ and $(b+1) \quad b \quad b$. Then $\widehat{HFL}(s_{1}, s_{2})\cong \F[-2(a+b)] \oplus \F[-2(a+b)-1]$ and $\widehat{HFK}(L_{1}, s_{1})\cong \F[-2a]$ and $\widehat{HFK}(L_{2}, s_{2})\cong \F[-2b]$. So $\widehat{HFL}(L, s_{1}, s_{2})\cong \widehat{HFK}(L_{1}, s_{1})\otimes \widehat{HFK}(L_{2}, s_{2})\otimes (\F\oplus\F_{(-1)})$. 

In Case $b$, the $h$-function for link component $L_{1}$ is like Case (2): $(a+1) \quad a \quad a$ and the $h$-function for $L_{2}$ is like Case (3): $b \quad b \quad b-1$.  In Case $c$, the $h$-function for $L_{1}$ is like  Case (3) and for $L_{2}$, $h$-function is like Case (2). In Case $d$, the $h$-functions for both components are like Case (3). Thus we can use similar argument above to prove that $\widehat{HFL}(L, s_{1}, s_{2})\cong \widehat{HFK}(L_{1}, s_{1})\otimes \widehat{HFK}(L_{2}, s_{2})\otimes (\F\oplus\F_{(-1)})$ in these cases.

If the $h$-function corresponding to $\widehat{HFL}(s_{1}, s_{2})$ is not in the above four cases, $\widehat{HFL}(s_{1}, s_{2})=0$, and at least one of $\widehat{HFK}(L_{1}, s_{1})$ and $\widehat{HFK}(L_{2}, s_{2})$ is zero. Thus the conclusion also holds. 
\end{proof}

\noindent
{\bf Proof of Theorem \ref{split}:}
Let $L=L_{1}\cup L_{2}$ be an $L$-space link with vanishing Alexander polynomial. The linking number of $L_{1}$ and $L_{2}$ is $0$ by Equation (2.1). By Theorem \ref{thm 2}, the Heegaard Floer link homology $\widehat{HFL}(s_{1}, s_{2})$ is determined by $\Delta_{L}(t_{1}, t_{2})$, $\Delta_{L_{1}}(t)$ and $\Delta_{L_{2}}(t)$. So $\widehat{HFL}(L, s_{1}, s_{2})\cong \widehat{HFL}(L_{1}\sqcup L_{2}, s_{1}, s_{2})\cong \widehat{HFK}(L_{1}, s_{1})\otimes \widehat{HFK}(L_{2}, s_{2})\otimes (\F\oplus\F_{(-1)}) $ with any $(s_{1}, s_{2})\in \bH$.

\begin{example}
The link Floer homology polytope for the split disjoint union of two right-handed trefoils.

Let $L=L_{1}\sqcup L_{2}$ be the split disjoint union of two right-handed trefoils. Recall that the right-handed trefoil is an $L$-space knot with Alexander polynomial $\Delta_{L_{1}}(t)=t-1+t^{-1}$, and 
$$\sum\limits_{s_{1}\in \Z} \chi (HFK^{-}(L_{1}, s_{1})) t^{s_{1}}=\dfrac{\Delta_{L_{1}}}{1-t^{-1}}=t+t^{-1}+t^{-2}+t^{-3}+t^{-4}+\cdots$$
Observe the short exact sequence $0\rightarrow A^{-}(s_{1}-1)\rightarrow A^{-1}(s_{1})\rightarrow CFK^{-}(s_{1})\rightarrow 0 $. We have $HFK^{-}(L_{1}, s_{1})=H_{\ast}(A^{-}(s_{1})/A^{-}(s_{1}-1))$, and $\chi(HFK^{-}(L_{1}, s_{1}))=h_{1}(s_{1}-1)-h_{•}(s_{1})$ which is also the coefficient of $t^{s_{1}}$ in $\dfrac{\Delta_{L_{1}}(t)}{1-t^{-1}}$. Since $L_{1}$ is an $L$-space knot. $h_{1}(s_{1})=0$ for sufficiently large $s_{1}\gg 0$. So the $h$-function $h_{1}(s_{1})$ can be determined as follows:
$$\cdots \quad 7, 6, 5, 4, 3, 2, 1, 1, 0, 0, 0, 0, 0, \cdots$$
where $h_{1}(0)=h_{1}(-1)=1$, $h_{1}(s)=0$ if $s\geq 1$ and $h_{1}(s)=-s$ if $s\leq -1$. Similarly, for another right-handed trefoil $L_{2}$, the $h$-function $h_{2}(s_{2})$ is the same with $h_{1}(s_{1})$. By Proposition 5.4,  we can find all $(s_{1}, s_{2})\in \bH$ where $\widehat{HFL}(L, s_{1}, s_{2})$ are nonzero. So $\widehat{HFL}(L, 1, 1)=\F[0]\oplus \F[-1]$, $\widehat{HFL}(L, 0, 1)=\widehat{HFL}(L, 1, 0)=\F[-1]\oplus \F[-2]$, $\widehat{HFL}(L, -1, 1)=\widehat{HFL}(L, 0, 0)=\widehat{HFL}(L, 1, -1)=\F[-2]\oplus \F[-3]$ and $\widehat{HFL}(L, -1, 0)=\widehat{HFL}(L, 0, -1)=\F[-3]\oplus \F[-4]$, $\widehat{HFL}(L, -1, -1)=\F[-4]\oplus \F[-5]$. For other lattice points $(s_{1}, s_{2})\in \bH$, $\widehat{HFL}(L, s_{1}, s_{2})=0$. Thus the link Floer homology polytope is a square in Figure 14.

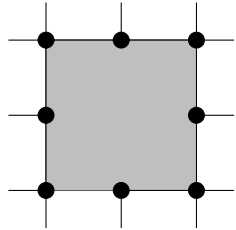
\begin{figure}[H]
\begin{tikzpicture}
\draw (-1.5,1)--(1.5, 1);
\draw (-1.5,-0)--(1.5,0);
\draw (-1.5,-1)--(1.5,-1);
\draw (-1,1.5)--(-1,-1.5);
\draw (0,1.5)--(0,-1.5);
\draw (1,1.5)--(1,-1.5);
\fill [fill=lightgray] 
(-1,-1)--(-1,0)--(-1,1)--(0,1)--(1,1)--(1,0)--(1,-1)--(0,-1);
\filldraw
(-1, -1) circle (3pt) -- (-1, 0) circle (3pt) -- (-1, 1) circle (3pt) -- (0, 1) circle (3pt) -- (1, 1) circle (3pt) -- (1, 0) circle (3pt) --(1, -1) circle (3pt) --(0, -1) circle (3pt);
 \end{tikzpicture}
\caption{ Link Floer homology polytope for $L$}
\label{L2}
\end{figure}

\end{example}

\begin{remark}
In general, let $L=L_{1}\sqcup L_{2}$ be the split union of any two $L$-space knots. Recall that the genus of a knot $K$ is defined as:
$$g(K)=\textup{min}\lbrace \textup{genus}(F) \mid F\subseteq S^{3} \textup{ is an oriented, embedded surface with } \partial F=K  \rbrace$$
Here $L_{1}$ and $L_{2}$ are both $L$-space knots, so $\widehat{HFK}(L_{1}, g(L_{1}))\cong \Z,   \widehat{HFK}(L_{2}, g(L_{2}))\cong \Z$  \cite[Theorem 1.2]{OS} and $g(L_{i})=\textup{max}\lbrace s\geq 0 \mid \widehat{HFK}_{\ast}(L_{i}, s)\neq 0 \rbrace$ for $i=1$ and $i=2$ \cite[Theorem 1.2]{os4}. The link Floer homology polytope of $L_{i}$ is the interval $[-g(L_{i}), g(L_{i})]$ where $i=1$ or $2$. By Proposition 5.4, the link Floer homology polytope for $L=L_{1}\sqcup L_{2}$ is a rectangle with vertices $(g(L_{1}), g(L_{2})), (g(L_{1}), -g(L_{2})), (-g(L_{1}), g(L_{2}))$ and $(-g(L_{1}), -g(L_{2}))$ (see Figure 14). 
\end{remark}


\begin{thebibliography}{99}

\bibitem{BG} M.Borodzik, E. Gorsky. Immersed concordances of links and Heegaard Floer homology. arXiv e-prints, 2016.

\bibitem{ND} N. Dawra. On the link Floer homology of $L$-space links. arXiv: 1505. 01100v1, 2015. 

\bibitem{GN} E. Gorsky, A. N\'emethi. Lattice and Heegaard-Floer homologies of algebraic links.
International Mathematics Research Notices, 2015, doi: 10.1093/imrn/rnv075

\bibitem{GN2} E. Gorsky, A. N\'emethi. Links of plane curve singularities are $L$-space link. Algebr. Geom. Topol. 16 (2016), no. 4, 1905-1912.  


\bibitem{Liu} Y. Liu. $L$-space surgeries on links. To appear in Quantum Topology.

\bibitem{MO} C. Manolescu and P. Ozsv\'ath. Heegaard Floer homology and integer surgeries on links. arXiv: 1011. 1317, 2010. 

\bibitem{Mu} C. T. McMullen. The Alexander polynomial of a 3-manifold of a 3-manifold and the Thurston norm on cohomology. Ann. Sci. de l'Ecole Norm. Sup., 35(2): 153-171, 2002. 

\bibitem{os4} P. Ozsv\'ath, Z. Szab\'o. Holomorphic disks and genus bounds, Geom. Topol. {\bf 8} (2004), 311-334.

\bibitem{OS} P. Ozsv\'ath, Z. Szab\'o. On knot Floer homology and lens space surgeries. Topology {\bf 44} (2005), no.6, 1281--1300.

\bibitem{os3} P. Ozsv\'ath, Z. Szab\'o. Heegaard diagram and Floer homology. International Congress of Mathematicians. Vol. II, 1083-1099, Eur. Math. Soc., Z$\ddot{u}$rich, 2006. 

\bibitem{os2} P. Ozsv\'ath, Z. Szab\'o. Holomorphic disks, link invariants and the multi-variable Alexander polynomial. Algebr. Geom. Topol. {\bf8} (2008), no. 2, 615-692. 

\bibitem{oss}  P.Ozsv\'ath,  Z.Szab\'o. Link Floer homology and the Thurston norm. J. Amer. Math. Soc. {\bf21} (2008) 671-709.

\bibitem{RR}  J. Remigio-Ju\'arez, Y. Rieck. The link volumes of some prism manifolds. Algebr. Geom. Topol. 12 (2012), no. 3, 1649-1665. 

\bibitem{Th}  W. P. Thurston. A norm for the homology of 3-manifolds. Volume {\bf59} of Mem. Amer. Math. Soc, pages 99-130. 1986.

\bibitem{Vafa} F. Vafaee.  On the knot Floer homology of twisted torus knots. Int. Math. Res. Not. IMRN 2015, no. 15, 6516-6537. 

\end{thebibliography}
\end{document}